\documentclass{birkjour_t2}
\usepackage{amsmath}
\usepackage{amsfonts}
\usepackage{amssymb}
\usepackage[english]{babel}
\usepackage{color}
\usepackage{graphicx}
\usepackage{tkz-euclide}
\usetkzobj{all}
\tikzset{elegant/.style={smooth,thick,samples=50,cyan}}
\tikzset{eaxis/.style={->,>=stealth}}
\usepackage{lmodern}
\usepackage{amsthm}
\usepackage{graphicx}
\usepackage{mathrsfs}
\usepackage{makecell}
\usepackage{microtype}
\usepackage{enumerate}
\usepackage{mathscinet}
\usepackage{array}
\usepackage{multirow}
\usepackage{enumerate}
\usepackage[cal=boondoxo,bb=ams]{mathalfa}
\usepackage{hyperref}
\hypersetup{hidelinks}

\renewcommand{\theequation}{\arabic{section}.\arabic{equation}}

\newtheorem{thm}{Theorem}[section]

\newtheorem{lem}{Lemma}[section]
\newtheorem{coro}{Corollary}[section]
\newtheorem{rem}{Remark}[section]

\newcommand{\ml}{\mathcal}
\newcommand{\mb}{\mathbb}
\DeclareMathOperator{\diag}{diag}

\DeclareMathOperator{\intt}{int}
\DeclareMathOperator{\extt}{ext}
\DeclareMathOperator{\midd}{mid}

\begin{document}

%-------------------------------------------------------------------------
% editorial commands: to be inserted by the editorial office
%
%\firstpage{1} \volume{228} \Copyrightyear{2004} \DOI{003-0001}
%\seriesextraline{This is the Concrete Title of this Book\br H.E. R and S.T.C. W, Eds.}
%\firstpage{1}
%      environment-name

\title[Thermoelastic plate equations with different damping mechanisms]
 {Cauchy problems for thermoelastic plate equations with different damping mechanisms}

\author[W. Chen]{Wenhui Chen}
\address{Institute of Applied Analysis, Faculty for Mathematics and Computer Science\\
	 Technical University Bergakademie Freiberg\\
	  Pr\"{u}ferstra{\ss}e 9\\
	   09596 Freiberg\\
	    Germany}
\email{Wenhui.Chen@student.tu-freiberg.de}

%\MSC[2010] Primary 35L52, 35L99; Secondary 35B33, 35B44

\date{January 1, 2004}

\begin{abstract}
In this paper we study Cauchy problem for thermoelastic plate equations with friction or structural damping in $\mb{R}^n$, $n\geq1$, where the heat conduction is modeled by Fourier's law. We explain some qualitative properties of solutions influenced by different damping mechanisms. We show which damping in the model has a dominant influence on smoothing effect, energy estimates, $L^p-L^q$ estimates not necessary on the conjugate line, and on diffusion phenomena. Moreover, we derive asymptotic profiles of solutions in a framework of weighted $L^1$ data. In particular, sharp decay estimates for lower bound and upper bound of solutions in the $\dot{H}^s$ norm ($s\geq0$) are shown.
\end{abstract}

%%% ----------------------------------------------------------------------
\maketitle
%%% ----------------------------------------------------------------------
%\tableofcontents
         \section{Introduction}\label{Section Introduction}

In recent years, thermoelastic plate equations have attracted a lot of attention. The Cauchy problem for linear thermoelastic plate equations are modeled by
\begin{align}\label{Eq. classical thermo plate Fourier}
\left\{
\begin{aligned}
&u_{tt}+\Delta^2u+\Delta\theta=0,&&t>0,\,\,x\in\mb{R}^n,\\
&\theta_t-\Delta\theta-\Delta u_{t}=0,&&t>0,\,\,x\in\mb{R}^n,\\
&(u,u_t,\theta)(0,x)=(u_0,u_1,\theta_0)(x),&&x\in\mb{R}^n,
\end{aligned}
\right.
\end{align}
where the unknowns $u=u(t,x)$ and $\theta=\theta(t,x)$ denote the elongation of a plate and the temperature difference to the equilibrium state, respectively. The recent papers \cite{SaidHouari2013,RackeUeda2016} investigated $L^2$-decay estimates of solutions to \eqref{Eq. classical thermo plate Fourier} by using energy methods in the Fourier space. Simultaneously, \cite{RackeUeda2016} proved the sharpness of the derived decay estimates by applying asymptotic expansions of eigenvalues. Other studies on thermoelastic plate equations can be found in the literature. We refer to \cite{AvalosLasiecka1997,Kim1992,LasieckaTriggiani1998,LasieckaTriggiani199802,LasieckaTriggiani199803,LasieckaTriggiani199804,LiuLiu1997,MunozRiveraRacke1995,MunozRiverRacke1996} for the initial boundary value problem in bounded domains, \cite{DenkRacke2006,DenkRackeShibata2009,DenkRackeShibata2010,MunozRiveraRacke1995,MunozRiverRacke1996} for the Cauchy problem or in general exterior domains.

In the real world and applications, due to some kinds of resistance in the elongation of a plate, we always model a plate equation with damping terms, for instance, plate equations with structural damping in \cite{HorbachIkehataCharao2016,IkehataSoga2014}. When the thermal dissipation modeled by Fourier's law and the dissipation for the elongation of a plate appear at the same time, we model thermoelastic plate equations with an additional damping in the equation for $u$, for example, the thermoelastic plate equations with friction $u_t$ presented in  \cite{WangWnag2018}, with structural damping $-\Delta u_t$ presented in  \cite{FatoriJorgeSilvaMaYang2015,Hao2008}, with Kelvin-Voigt type damping or viscoelastic damping $\Delta^2u_t$ presented in  \cite{WangZhang2017}. For this reason, we consider thermoelastic plate equations with different damping mechanisms in present paper. 

In this paper we are concerned with the following Cauchy problem for thermoelastic plate equations in $\mb{R}^n$, $n\geq1$, where the heat conduction is modeled by Fourier's law:
\begin{align}\label{Eq. thermo plate Fourier}
\left\{
\begin{aligned}
&u_{tt}+\Delta^2u+\Delta\theta+(-\Delta)^{\sigma}u_t=0,&&t>0,\,\,x\in\mb{R}^n,\\
&\theta_t-\Delta\theta-\Delta u_{t}=0,&&t>0,\,\,x\in\mb{R}^n,\\
&(u,u_t,\theta)(0,x)=(u_0,u_1,\theta_0)(x),&&x\in\mb{R}^n,
\end{aligned}
\right.
\end{align}
where $\sigma\in[0,2]$. To be more specific, $\sigma=0$ stands for the system with \emph{friction or external damping}, $\sigma\in(0,2]$ stands for the system with \emph{structural damping}, especially, $\sigma=2$ stands for the system with \emph{Kelvin-Voigt type damping}. The example of the model of thermoealstic plate equations with friction or structural damping \eqref{Eq. thermo plate Fourier} is a special case of $\alpha-\beta-\gamma$ systems, which have been introduced in \cite{FischerRacke2018},
\begin{align*}
\left\{
\begin{aligned}
&u_{tt}+\ml{A}u-\ml{A}^{\beta}\theta+\ml{A}^{\gamma}u_t=0,\\
&\theta_t+\ml{A}^{\alpha}\theta+\ml{A}^{\beta}u_t=0,\\
\end{aligned}
\right.
\end{align*}
when we choose
\begin{align*}
\ml{A}=(-\Delta)^2\text{ and }\alpha=\beta=\frac{1}{2},\,\,\,\,[0,1]\ni\gamma=\frac{1}{2}\sigma.
\end{align*}
Nevertheless, the influence of an additional damping in the equation for $u$ in thermoelastic plate equations on some qualitative properties of solutions as $L^p-L^q$ estimates, diffusion phenomena, asymptotic profiles of solutions is still open.

Our main purpose of this paper is to study different qualitative properties of solutions to the thermoelastic plate equations with different damping mechanisms. More specifically, we are interested in the following properties of solutions to \eqref{Eq. thermo plate Fourier}:
\vskip2mm
\begin{enumerate}[(1)]
	\item smoothing effect and $L^2$ well-posedness;
	\item energy estimates with different assumptions on initial data;
	\item $L^p-L^q$ estimates not necessary on the conjugate line;
	\item diffusion phenomena;
	\item asymptotic profiles of solutions in a framework of weighted $L^1$ data.
\end{enumerate}
\vskip2mm
Then, due to the fact that different kinds of damping (friction, structural damping, thermal damping) have different influence on the model, we will analyze the dominant influence from the damping to different qualitative properties of solutions. In other words, there exists a competition between ``friction or structural damping" and ``thermal damping generated by Fourier's law".
\vskip2mm

In order to study the above qualitative properties of solutions, especially, $L^p-L^p$ estimates for $1\leq p\leq \infty$, diffusion phenomena and asymptotic profiles of solutions, we need to derive representations of solutions instead of using pointwise estimates in the Fourier space. However, because the fractional power operator $(-\Delta)^{\sigma}$ acts on $u_t$ in the damping term, the method of asymptotic expansions of eigenprojections (c.f. \cite{IdeHaramotoKawashima2008,Chen2018}) seems to be difficult to applied. Moreover, the method of asymptotic expansions of eigenvalues (c.f. \cite{RackeUeda2016,IdeHaramotoKawashima2008}) cannot be used directly to prove the sharpness for the derived estimates of solutions. To overcome these difficulties, we may derive representations of solutions by applying methods of WKB analysis. The main tool is the application of a multi-step diagonalization procedure, which was mainly proposed in \cite{JachmannReissig2008,ReissigWang2005}.

In the study of diffusion phenomena of solutions to \eqref{Eq. thermo plate Fourier}, we may observe the corresponding reference system to \eqref{Eq. thermo plate Fourier}, which is consisted of different evolution equations, e.g., heat equation, fractional heat equation, Schr\"odinger equation, fourth-order parabolic equation. The equations of such reference system are determined by the value of $\sigma$ in the damping term $(-\Delta)^{\sigma}u_t$.  It will provide some opportunities for us to understand the model \eqref{Eq. thermo plate Fourier} in a more precise way. 

For asymptotic profiles of solutions in a framework of weighted $L^1$ data in Section \ref{Section Asymptotic behavior}, by introducing
\begin{align*}
U(t,x):&=\left(u_t+|D|^2u,u_t-|D|^2u,\theta\right)^{\mathrm{T}}(t,x),\\
U_0(x):&=\left(u_1+|D|^2u_0,u_1-|D|^2u_0,\theta_0\right)^{\mathrm{T}}(x),\\
P_{U_0}:&=\int_{\mb{R}^n}U_0(x)dx\,\,\,\,\,\,\text{with}\,\,\,\,\,\,|P_{U_0}|\neq0,
\end{align*}
we will prove the following estimates for $t\gg1$:
\begin{align*}
t^{-\frac{n+2s}{4\max\{2-\sigma;1\}}}\left|P_{U_0}\right|\lesssim\left\|U(t,\cdot)\right\|_{\dot{H}^s(\mb{R}^n)}\lesssim t^{-\frac{n+2s}{4\max\{2-\sigma;1\}}}\|U_0\|_{H^s(\mb{R}^n)\cap L^{1,1}(\mb{R}^n)},
\end{align*}
with  $s\geq0$, $\sigma\in[0,2]$. It immediately leads to sharp decay rate of the estimates for the $\dot{H}^s$ norm of solutions.
To the best of the author's knowledge, sharp estimates for lower bound of solutions for dissipative elastic systems are unknown, although the estimates for upper bound of solutions in the $L^2$ norm have been extensively discussed in several kinds of elastic systems, see \cite{Reissig2016,ChenReissig2018,Chen2018} for dissipative elastic waves, \cite{JachmannReissig2008,ReissigWang2005,WangYang2006,YangWang2006,YangWang200602} for thermoelastic systems. We remark that our method can probably applied to some other systems in elastic material (see Remark \ref{Generalization}), too. Furthermore, due to the double damping, including friction or structural damping and thermal damping generated by Fourier's law, it is interesting to investigate which damping will give stronger effects on the asymptotic profiles of solutions.

The paper is organized as follows. In Section \ref{Subsec Fourier asym behav}, we prepare representations of solutions to \eqref{Eq. thermo plate Fourier} by WKB analysis. In Section \ref{Section Qualitative properties}, by using these representations of solutions we study smoothing effect of solutions and $L^2$ well-posedness of the Cauchy problem \eqref{Eq. thermo plate Fourier}. In Section \ref{Subsec Fouier estimate}, we derive some estimates of solutions, including energy estimates with initial data taking from $H^s(\mb{R}^n)\cap L^m(\mb{R}^n)$ with $s\geq0$, $m\in[1,2]$, and $L^p-L^q$ estimates not necessary on the conjugate line. In Section \ref{Subsec Fourier asym profile}, diffusion phenomena for linear thermoelastic plate equations with friction or structural damping are investigated. In Section \ref{Section Asymptotic behavior}, we derive long-time asymptotic profiles of solutions in a framework of weighted $L^1$ data. Finally, in Section \ref{Section Concluding remarks} some concluding remarks complete the paper.

\vskip2mm

\noindent\textbf{Notation. } We now give some notations to be used in this paper.  We denote the identity matrix of dimension $k\times k$ by $I_{k}$. $f\lesssim g$ means that there exists a positive constant $C$ such that $f\leq Cg$. Moreover, $H^s_p(\mb{R}^n)$ and $\dot{H}^s_p(\mb{R}^n)$ with $s\geq0$ and $1\leq p<\infty$, denote Bessel and Riesz potential spaces based on $L^p(\mb{R}^n)$, respectively. Here $\langle D\rangle^s$ and $|D|^s$ stand for the pseudo-differential operators with symbols $\langle\xi\rangle^s$ and $|\xi|^s$, respectively. 

Let us define the Gevrey spaces $\Gamma^{\kappa}(\mb{R}^n)$ for $\kappa\in[1,\infty)$ by (c.f.  \cite{Rodino1993})
\begin{align*}
\Gamma^{\kappa}(\mb{R}^n):=\left\{f\in L^2(\mb{R}^n):\text{ there exists a constant } c\text{ such that } \exp\left(c\langle\xi\rangle^{\frac{1}{\kappa}}\right)\hat{f}\in L^2(\mb{R}^n)\right\}.
\end{align*}

Let us define the weighted $L^1$ spaces $L^{1,\delta}(\mb{R}^n)$ for $\delta\in[0,\infty)$ by
\begin{align*}
L^{1,\delta}(\mb{R}^n):=\left\{f\in L^1(\mb{R}^n):\|f\|_{L^{1,\delta}(\mb{R}^n)}:=\int_{\mb{R}^n}(1+|x|)^{\delta}|f(x)|dx\right\}.
\end{align*}
Particularly, we notice that $L^{1,0}(\mb{R}^n)=L^1(\mb{R}^n)$.
\section{Asymptotic behavior of solutions}\label{Subsec Fourier asym behav}

First of all, we apply the partial Fourier transformation with respect to spatial variables to \eqref{Eq. thermo plate Fourier} to get the following second-order ordinary differential system:
\begin{align}\label{Eq. Four. thermo plate Fourier}
\left\{
\begin{aligned}
&\hat{u}_{tt}+|\xi|^4\hat{u}-|\xi|^2\hat{\theta}+|\xi|^{2\sigma}\hat{u}_t=0,&&t>0,\,\,\xi\in\mb{R}^n,\\
&\hat{\theta}_t+|\xi|^2\hat{\theta}+|\xi|^2\hat{u}_{t}=0,&&t>0,\,\,\xi\in\mb{R}^n,\\
&\big(\hat{u},\hat{u}_t,\hat{\theta}\big)(0,\xi)=\big(\hat{u}_0,\hat{u}_1,\hat{\theta}_0\big)(\xi),&&\xi\in\mb{R}^n.
\end{aligned}
\right.
\end{align}
Introducing a new ansatz $w^{(0)}=w^{(0)}(t,\xi)$ by
\begin{align*}
w^{(0)}:=\left(\hat{u}_t+|\xi|^2\hat{u},\hat{u}_t-|\xi|^2\hat{u},\hat{\theta}\right)^{\mathrm{T}},
\end{align*}
we obtain the first-order system as follows:
\begin{align}\label{Eq. first-order Fourier law}
\left\{
\begin{aligned}
&w_t^{(0)}+\left(|\xi|^2A_0+|\xi|^{2\sigma}A_1\right)w^{(0)}=0,&&t>0,\,\,\xi\in\mb{R}^n,\\
&w^{(0)}(0,\xi)=w^{(0)}_0(\xi),&&\xi\in\mb{R}^n,
\end{aligned}
\right.
\end{align}
where $w^{(0)}_0=w^{(0)}_0(\xi)$ is defined by
\begin{align*}
w^{(0)}_0:=\left(\hat{u}_1+|\xi|^2\hat{u}_0,\hat{u}_1-|\xi|^2\hat{u}_0,\hat{\theta}_0\right)^{\mathrm{T}},
\end{align*}
and the coefficient matrices are given by
\begin{align*}
A_0=\frac{1}{2}\left(
{\begin{array}{*{20}c}
	0 & -2 & -2\\
	2 & 0 & -2\\
	1 & 1 & 2
	\end{array}}
\right) \,\,\,\, \text{and} \,\,\,\, A_1=\frac{1}{2}\left(
{\begin{array}{*{20}c}
	1 & 1 & 0\\
	1 & 1 & 0\\
	0 & 0 & 0
	\end{array}}
\right).
\end{align*}
Additionally, we denote the matrix \begin{align*}
A(|\xi|;\sigma):=|\xi|^2A_0+|\xi|^{2\sigma}A_1,
\end{align*}
and
\begin{align}
U(t,x):&=\left(u_t+|D|^2u,u_t-|D|^2u,\theta\right)^{\mathrm{T}}(t,x),\label{Solution}\\
U_0(x):&=\left(u_1+|D|^2u_0,u_1-|D|^2u_0,\theta_0\right)^{\mathrm{T}}(x).\label{Initialdata}
\end{align}
It is clear that $\ml{F}_{x\rightarrow \xi}\left(U(t,x)\right)=w^{(0)}(t,\xi)$ and $\ml{F}\left(U_0(x)\right)=w^{(0)}_0(\xi)$.
\subsection{Diagonalization schemes}
In the beginning, we divide the phase space into three regions
\begin{align*}
Z_{\intt}(\varepsilon)&=\left\{\xi\in\mb{R}^n:|\xi|\leq\varepsilon\ll1\right\},\\
Z_{\midd}(\varepsilon,N)&=\left\{\xi\in\mb{R}^n:\varepsilon\leq |\xi|\leq N\right\},\\
Z_{\extt}(N)&=\left\{\xi\in\mb{R}^n:|\xi|\geq N\gg1\right\},
\end{align*}
for small, bounded and large frequencies. Later, we  will diagonalize the principle part of the first-order system \eqref{Eq. first-order Fourier law} in each region. Furthermore, let us define $\chi_{\intt}(\xi),\chi_{\midd}(\xi),\chi_{\extt}(\xi)\in \mathcal{C}^{\infty}(\mb{R}^n)$ having their supports in $Z_{\intt}(\varepsilon)$, $Z_{\midd}(\varepsilon,N)$ and $Z_{\extt}(N)$, respectively, so that $\chi_{\midd}(\xi)=1-\chi_{\intt}(\xi)-\chi_{\extt}(\xi)$.

To understand the influence of the parameter $\xi$ on the asymptotic behavior of solutions, we now distinguish between the next four cases.
\begin{itemize}
	\item Case 2.1: $\sigma\in[0,1)$ with $\xi\in Z_{\intt}(\varepsilon)$ or $\sigma\in(1,2]$ with $\xi\in Z_{\extt}(N)$;
	\item Case 2.2: $\sigma\in[0,1)$ with $\xi\in Z_{\extt}(N)$ or $\sigma\in(1,2]$ with $\xi\in Z_{\intt}(\varepsilon)$;
	\item Case 2.3: $\sigma=1$ for all frequencies;
	\item Case 2.4: $\sigma\neq1$ with $\xi\in Z_{\midd}(\varepsilon,N)$.
\end{itemize}
For frequencies in the small zone or the large zone, i.e., Cases 2.1 and 2.2, the diagonalization procedure is available. This procedure, which is developed \cite{JachmannReissig2008,ReissigWang2005,Yagdjian1997}, allows us to derive representations of solutions. For Case 2.3, the matrix $A(|\xi|;1)$ may be understood as no perturbed linear operator for all frequencies due to $A(|\xi|;1)=|\xi|^2(A_0+A_1)$. For this reason, we only need to calculate the eigenvalues of the matrix $A(|\xi|;1)$ directly in Case 2.3. For frequencies in the bounded zone and $\sigma\neq1$, i.e., Case 2.4, we construct a contradiction to prove that the real parts of the characteristic roots have a fixed sign. 

\vskip2mm

\begin{lem}[Treatment for Case 2.1]\label{Lemma 2.1}
	When $\sigma\in[0,1)$ with $\xi\in Z_{\intt}(\varepsilon)$, or $\sigma\in(1,2]$ with $\xi\in Z_{\extt}(N)$, after $\ell$ steps of diagonalization procedure the starting system \eqref{Eq. Four. thermo plate Fourier} is transformed to
	\begin{align*}
	\left\{
	\begin{aligned}
	&w_t^{(\ell)}-\left(\Lambda_0+\cdots+\Lambda_{\ell}+R_{\ell+1}\right)w^{(\ell)}=0,&&t>0,\,\,\xi\in\mb{R}^n,\\
	&w^{(\ell)}(0,\xi)=w_0^{(\ell)}(\xi),&&\xi\in\mb{R}^n,
	\end{aligned}
	\right.
	\end{align*}
	with the diagonalized matrices $\Lambda_1,\dots,\Lambda_{\ell}$ and the remainder $R_{\ell+1}$. The asymptotic  behavior of these matrices can be described as follows:
	\begin{align*}
	\Lambda_0=\ml{O}\big(|\xi|^{2\sigma}\big),\,\,\,\,\Lambda_j=\ml{O}\big(|\xi|^{2(1-\sigma)(j-1)+2}\big),\,\,\,\,R_{\ell+1}=\ml{O}\big(|\xi|^{2(1-\sigma)\ell+2}\big).
	\end{align*}
	Moreover, the characteristic roots $\lambda_{\ell,j}=\lambda_{\ell,j}(|\xi|)$ with $j=1,2,3,$ having the following asymptotic behavior:
	\begin{align*}
	\lambda_{\ell,1}=|\xi|^{4-2\sigma},\,\,\,\,\lambda_{\ell,2}=|\xi|^{2}+|\xi|^{4-2\sigma},\,\,\,\,\lambda_{\ell,3}=|\xi|^{2\sigma}-2|\xi|^{4-2\sigma},
	\end{align*} 
	modulo $\ml{O}\big(|\xi|^{6-4\sigma}\big)$. 
\end{lem}
\begin{proof}
	To start the diagonalization procedure, the matrix $|\xi|^{2\sigma}A_1$ has a dominant influence in comparison with the matrix $|\xi|^2A_0$ in Case 2.1. As the consequence, we should diagonalize $|\xi|^{2\sigma}A_1$ in the first place. With the aid of variable change
	\begin{align*}
	w^{(1)}:=T_0^{-1}w^{(0)}:=\left(
	{\begin{array}{*{20}c}
		-1 & 0 & 1\\
		1 & 0 & 1\\
		0 & 1 & 0
		\end{array}}
	\right)^{-1}w^{(0)},
	\end{align*}
	we derive 
	\begin{align*}
	w_t^{(1)}+\left(\Lambda_0+R_1\right)w^{(1)}=0,
	\end{align*}
	where \begin{align*}
	\Lambda_0=\diag\big(0,0,|\xi|^{2\sigma}\big)=\ml{O}\big(|\xi|^{2\sigma}\big)
	\end{align*} and $R_1=|\xi|^2A_0^{(1)}=\ml{O}\big(|\xi|^2\big)$ with
	\begin{align*}
	A_0^{(1)}=T_0^{-1}A_0T_0=\left(
	{\begin{array}{*{20}c}
		0 & 0 & 1\\
		0 & 1 & 1\\
		-1 & -1 & 0
		\end{array}}
	\right)=\ml{O}(1).
	\end{align*}
	Next, we define \begin{align*}w^{(2)}:=T_1^{-1}w^{(1)}\end{align*} with $T_1:=I_{3}+N_1(|\xi|)$, where
	\begin{align*}
	N_1(|\xi|):=|\xi|^{2-2\sigma}\left(
	{\begin{array}{*{20}c}
		0 & 0 & 1\\
		0 & 0 & 1\\
		1 & 1 & 0
		\end{array}}
	\right)=\ml{O}\big(|\xi|^{2-2\sigma}\big).
	\end{align*}
	Thus, we have
	\begin{align}\label{Eq. Sup 1}
	w_t^{(2)}+\left(\Lambda_0+T_1^{-1}\left(|\xi|^2A_0^{(1)}-[N_1(|\xi|),\Lambda_0]\right)+|\xi|^2T_1^{-1}A_0^{(1)}N_1(|\xi|)\right)w^{(2)}=0,
	\end{align}
	where we use 
	\begin{align*}
	[N_1(|\xi|),\Lambda_0]:=N_1(|\xi|)\Lambda_0-\Lambda_0N_1(|\xi|)=|\xi|^{2}\left(
	{\begin{array}{*{20}c}
		0 & 0 & 1\\
		0 & 0 & 1\\
		-1 & -1 & 0
		\end{array}}
	\right).
	\end{align*}
	By the relation $T_1^{-1}=I_3-T_1^{-1}N_1(|\xi|)$, we transform \eqref{Eq. Sup 1} to the following first-order system:
	\begin{align*}
	w_t^{(2)}+\left(\Lambda_0+\Lambda_1+R_2\right)w^{(2)}=0,
	\end{align*}
	where \begin{align*}
	\Lambda_1:=\diag\big(0,|\xi|^2,0\big)=\ml{O}\big(|\xi|^2\big)
	\end{align*} and $R_2=A_0^{(2)}-T_1^{-1}N_1(|\xi|)A_0^{(2)}=\ml{O}\big(|\xi|^{4-2\sigma}\big)$ with
	\begin{align*}
	A_0^{(2)}=-N_1(|\xi|)\Lambda_1+|\xi|^2A_0^{(1)}N_1(|\xi|)=|\xi|^{4-2\sigma}\left(
	{\begin{array}{*{20}c}
		1 & 1 & 0\\
		1 & 1 & 1\\
		0 & -1 & -2
		\end{array}}
	\right)=\ml{O}\big(|\xi|^{4-2\sigma}\big).
	\end{align*}
	
	By the similar procedure, we introduce \begin{align*}
	w^{(3)}:=T_2^{-1}T_{1\frac{1}{2}}^{-1}w^{(2)}
	\end{align*} with $T_{1\frac{1}{2}}:=I_3+N_{1\frac{1}{2}}(|\xi|)$ and $T_{2}:=I_3+N_{2}(|\xi|)$, where
	\begin{align*}
	&N_{1\frac{1}{2}}(|\xi|):=|\xi|^{4-4\sigma}\left(
	{\begin{array}{*{20}c}
		0 & 0 & 0\\
		0 & 0 & 1\\
		0 & 1 & 0
		\end{array}}
	\right)=\ml{O}\big(|\xi|^{4-4\sigma}\big),\\
	&N_{2}(|\xi|):=|\xi|^{2-2\sigma}\left(
	{\begin{array}{*{20}c}
		0 & 1 & 0\\
		-1 & 0 & 0\\
		0 & 0 & 0
		\end{array}}
	\right)=\ml{O}\big(|\xi|^{2-2\sigma}\big).
	\end{align*}
	So, we derive the following system:
	\begin{align*}
	w_t^{(3)}+\left(\Lambda_0+\Lambda_1+\Lambda_2+R_{3}\right)w^{(3)}=0,
	\end{align*}
	where
	\begin{align*}
	\Lambda_2=\diag\big(|\xi|^{4-2\sigma},|\xi|^{4-2\sigma},-2|\xi|^{4-2\sigma}\big)=\ml{O}\big(|\xi|^{4-2\sigma}\big)
	\end{align*} and $R_3=\ml{O}\big(|\xi|^{6-4\sigma}\big)$. Notice that the matrices $T_{\sigma,\intt}$ and $T_{\sigma,\extt}$, respectively, are defined by
	\begin{align}
	T_{\sigma,\intt}:&=T_0T_1T_{1\frac{1}{2}}T_2\,\,\,\,\text{if}\,\,\,\,\sigma\in[0,1),\nonumber\\
	T_{\sigma,\extt}:&=T_0T_1T_{1\frac{1}{2}}T_2\,\,\,\,\text{if}\,\,\,\,\sigma\in(1,2].\label{T matrices Case 2.1}
	\end{align}
	Then, we carry out further steps of diagonalization proposed in \cite{ReissigSmith2005,Yagdjian1997} to complete the proof.
\end{proof}
\begin{lem}[Treatment for Case 2.2] \label{Lemma 2.2}
	When $\sigma\in[0,1)$ with $\xi\in Z_{\extt}(N)$, or $\sigma\in(1,2]$ with $\xi\in Z_{\intt}(\varepsilon)$, after $\ell$ steps of diagonalization procedure the starting system \eqref{Eq. Four. thermo plate Fourier} is transformed to
	\begin{align*}
	\left\{
	\begin{aligned}
	&w_t^{(\ell)}-\left(\Lambda_0+\cdots+\Lambda_{\ell}+R_{\ell+1}\right)w^{(\ell)}=0,&&t>0,\,\,\xi\in\mb{R}^n,\\
	&w^{(\ell)}(0,\xi)=w_0^{(\ell)}(\xi),&&\xi\in\mb{R}^n,
	\end{aligned}
	\right.
	\end{align*}
	with the diagonalized matrices $\Lambda_1,\dots,\Lambda_{\ell}$ and the remainder $R_{\ell+1}$. The asymptotic  behavior of these matrices can be described as follows:
	\begin{align*}
	\Lambda_0=\ml{O}\big(|\xi|^2\big),\,\,\,\,\Lambda_j=\ml{O}\big(|\xi|^{2(\sigma-1)(j-1)+2\sigma}\big),\,\,\,\,R_{\ell+1}=\ml{O}\big(|\xi|^{2(\sigma-1)\ell+2\sigma}\big).
	\end{align*}
	Moreover, the characteristic roots $\lambda_{\ell,j}=\lambda_{\ell,j}(|\xi|)$ with $j=1,2,3,$ having the following asymptotic behavior:
	\begin{align*}
	\lambda_{\ell,1}=y_1|\xi|^2,\,\,\,\,\lambda_{\ell,2}=y_2|\xi|^2,\,\,\,\,\lambda_{\ell,3}=y_3|\xi|^2,
	\end{align*} 
	modulo $\ml{O}\big(|\xi|^{2\sigma}\big)$, where the constants $y_j$ for $j=1,2,3$ have been shown in \eqref{Value y1 y2 y3} later. 
\end{lem}
\begin{proof}
	In this part, the matrix $|\xi|^2A_0$ has a dominant influence in comparison with the matrix $|\xi|^{2\sigma}A_1$. Thus, we should diagonalize $|\xi|^2A_0$ firstly. After applying the substitution \begin{align*}w^{(1)}:=T_0^{-1}w^{(0)},\end{align*} we arrive at the system
	\begin{align*}
	w_t^{(1)}+\left(\Lambda_0+R_1\right)w^{(1)}=0,
	\end{align*}
	with the diagonal matrix \begin{align*}
	\Lambda_0=|\xi|^2T_0^{-1}A_0T_0=\diag\big(y_1|\xi|^2,y_2|\xi|^2,y_3|\xi|^2\big)=\ml{O}\big(|\xi|^2\big)
	\end{align*} and the remainder $R_1=|\xi|^{2\sigma}T_0^{-1}A_1T_0=\ml{O}\big(|\xi|^{2\sigma}\big)$. In the above, the value of $y_j$ for $j=1,2,3$ are the solutions to the cubic equation \begin{align*}
	y^3-y^2+2y-1=0.
	\end{align*} Then, from direct calculations the value $y_j$ for $j=1,2,3,$ are given by
	\begin{align}\label{Value y1 y2 y3}
	y_1=\frac{1}{3}\left(1+z_1\right),\,\,\,\,y_2=\frac{1}{3}\left(1-\frac{1}{2}z_1+\frac{\sqrt{3}}{2}iz_2\right),\,\,\,\,y_3=\frac{1}{3}\left(1-\frac{1}{2}z_1-\frac{\sqrt{3}}{2}iz_2\right),
	\end{align}
	where
	\begin{align*}
	z_1=\sqrt[3]{\frac{1}{2}\left(3\sqrt{69}+11\right)}-\sqrt[3]{\frac{1}{2}\left(3\sqrt{69}-11\right)},\,\,\,\,z_2=\sqrt[3]{\frac{1}{2}\left(3\sqrt{69}+11\right)}+\sqrt[3]{\frac{1}{2}\left(3\sqrt{69}-11\right)}.
	\end{align*}
	Note that $y_1\neq y_2\neq y_3$ and the real part of $y_j$ are positive for all $j=1,2,3$. We now denote the matrices $T_{\sigma,\intt}$ and $T_{\sigma,\extt}$, respectively, by
	\begin{align}
	T_{\sigma,\intt}:&=T_0\,\,\,\,\text{if}\,\,\,\,\sigma\in(1,2],\nonumber\\
	T_{\sigma,\extt}:&=T_0\,\,\,\,\text{if}\,\,\,\,\sigma\in[0,1).\label{T matrices Case 2.2}
	\end{align}
	Finally, one may apply further steps of diagonalization proposed in \cite{ReissigSmith2005,Yagdjian1997} to complete the proof.
\end{proof}

\vskip2mm

\begin{lem}[Treatment for Case 2.3] \label{Lemma 2.3}
	When $\sigma=1$ with $\xi\in\mb{R}^n$, the starting system \eqref{Eq. Four. thermo plate Fourier} can be transformed to
	\begin{align*}
	\left\{
	\begin{aligned}
	&w_t^{(1)}-\Lambda_0w^{(1)}=0,&&t>0,\,\,\xi\in\mb{R}^n,\\
	&w^{(1)}(0,\xi)=w_0^{(1)}(\xi),&&\xi\in\mb{R}^n,
	\end{aligned}\right.
	\end{align*}
	with the diagonalize matrix $\Lambda_0=\diag\big(y_4|\xi|^2,y_5|\xi|^2,y_6|\xi|^2\big)$, where the constants $y_j$ for $j=4,5,6$ have been shown in \eqref{Value y4 y5 y6} later.
\end{lem}
\begin{proof}
	Here the matrices $|\xi|^2A_0$ and $|\xi|^{2\sigma}A_1$ with $\sigma=1$ have the same influence on the principle part. For this reason, the following system comes:
	\begin{align*}
	w^{(0)}_t+A(|\xi|;1)w^{(0)}=0.
	\end{align*}
	From direct calculation, we get
	\begin{align*}
	0=\text{det}\left(A(|\xi|;1)-\lambda I_{3}\right)&=\left|
	{\begin{array}{*{20}c}
		\frac{1}{2}|\xi|^{2}-\lambda & \frac{1}{2}|\xi|^{2}-|\xi|^2 & -|\xi|^2\\
		\frac{1}{2}|\xi|^{2}+|\xi|^2 & \frac{1}{2}|\xi|^{2}-\lambda & -|\xi|^2\\
		\frac{1}{2}|\xi|^2 & \frac{1}{2}|\xi|^2 & |\xi|^2-\lambda
		\end{array}}
	\right|\\
	&=-\lambda^3+2|\xi|^2\lambda^2-3|\xi|^4\lambda+|\xi|^6\\
	&=-|\xi|^6\left(\left(\frac{\lambda}{|\xi|^2}\right)^3-2\left(\frac{\lambda}{|\xi|^2}\right)^2+3\left(\frac{\lambda}{|\xi|^2}\right)-1\right).
	\end{align*}
	In other words, we only need to study the solution to the cubic equation \begin{align*}
	y^3-2y^2+3y-1=0.
	\end{align*} By a simple calculation, we find the solution to above cubic equation given by
	\begin{align}\label{Value y4 y5 y6}
	y_4=z_3-\frac{5}{9z_3}+\frac{2}{3},\,\,\,\,y_5=z_4-\frac{5}{9z_4}+\frac{2}{3},\,\,\,\,y_6=z_5-\frac{5}{9z_5}+\frac{2}{3},
	\end{align}
	where 
	\begin{align*}
	z_3=\frac{1}{3}\sqrt[3]{-\frac{11}{2}+\frac{3}{2}\sqrt{69}},\,\,\,\,z_{4}=\left(-\frac{1}{2}+\frac{\sqrt{3}}{2}i\right)z_3,\,\,\,\,z_{5}=\left(-\frac{1}{2}-\frac{\sqrt{3}}{2}i\right)z_3.
	\end{align*}
	Note that $y_4\neq y_5\neq y_6$ with $\text{Re } y_j>0$ for $j=4,5,6$.
	
	By introducing a new ansatz \begin{align*}
	w^{(1)}:=T_{1,0}^{-1}w^{(0)},
	\end{align*} we obtain the following system:
	\begin{align*}
	w^{(1)}_t+\Lambda_0w^{(1)}=0,
	\end{align*}
	with the diagonal matrix
	\begin{align*}
	\Lambda_0=T_{1,0}^{-1}A(|\xi|;1)T_{1,0}=\diag\left(y_4|\xi|^2,y_5|\xi|^2,y_6|\xi|^2\right)=\ml{O}\big(|\xi|^2\big).
	\end{align*}
	Then, the proof of this lemma is completed.
\end{proof}

Lastly, we derive an exponential decay result for frequencies in the bounded zone $Z_{\midd}(\varepsilon,N)$ to guarantee the stability of solutions to \eqref{Eq. first-order Fourier law} for $\sigma\in[0,1)\cup(1,2]$.

\vskip2mm

\begin{lem}[Treatment for Case 2.4] \label{Lemma 2.4}
	The solution $w^{(0)}=w^{(0)}(t,\xi)$ to the Cauchy problem \eqref{Eq. first-order Fourier law} with $\sigma\in[0,1)\cup(1,2]$ satisfies
	\begin{align*}
	\big|w^{(0)}(t,\xi)\big|\lesssim e^{-ct}\big|w_0^{(0)}(\xi)\big|,
	\end{align*}
	for $\xi\in Z_{\midd}(\varepsilon,N)$, where $c$ is a positive constant.
\end{lem}
\begin{proof}
	The following considerations help us obtain a priori estimate for the characteristic roots for frequencies in the bounded zone $Z_{\midd}(\varepsilon,N)$. We assume that there is a purely imaginary eigenvalue $\lambda=ia$ with $a\in\mb{R}\backslash\{0\}$ of the coefficient matrix $A(|\xi|;\sigma)$ for $\xi\neq0$. The eigenvalue $\lambda$ satisfies the following cubic equation:
	\begin{align}
	0=\text{det}(A(|\xi|;\sigma)-\lambda I_{3})&=\left|
	{\begin{array}{*{20}c}
		\frac{1}{2}|\xi|^{2\sigma}-\lambda & \frac{1}{2}|\xi|^{2\sigma}-|\xi|^2 & -|\xi|^2\\
		\frac{1}{2}|\xi|^{2\sigma}+|\xi|^2 & \frac{1}{2}|\xi|^{2\sigma}-\lambda & -|\xi|^2\\
		\frac{1}{2}|\xi|^2 & \frac{1}{2}|\xi|^2 & |\xi|^2-\lambda
		\end{array}}
	\right|\nonumber\\
	&=-\lambda^3+\big(|\xi|^{2\sigma}+|\xi|^2\big)\lambda^2-\big(2|\xi|^4+|\xi|^{2+2\sigma}\big)\lambda+|\xi|^6.\label{Supp 01}
	\end{align}
	Plugging $\lambda=ia$ in \eqref{Supp 01} and considering the real and imaginary parts of the coefficient of $a$, we conclude the following two equations, respectively:
	\begin{align*}
	\left\{
	\begin{aligned}
	-a^2\big(|\xi|^{2\sigma}+|\xi|^2\big)+|\xi|^6&=0,\\
	a\big(a^2-2|\xi|^4-|\xi|^{2+2\sigma}\big)&=0,
	\end{aligned}\right.
	\quad\Rightarrow\quad
	\left\{
	\begin{aligned}
	a^2&=\frac{|\xi|^6}{|\xi|^{2\sigma}+|\xi|^2},\\
	a^2&=2|\xi|^{4}+|\xi|^{2+2\sigma},
	\end{aligned}\right.
	\end{align*}
	where we use $a\neq0$.	They lead to a contradiction immediately because $\xi\in Z_{\midd}(\varepsilon,N)$. Then, no pure imaginary characteristic roots of $A(|\xi|;\sigma)$ for all $\sigma\in[0,1)\cup(1,2]$ can exist for frequencies in the bounded zone. Consequently, due to the compactness of the bounded zone $Z_{\midd}(\varepsilon,N)$ and the continuity of $\text{Re }\lambda_j(|\xi|)$ together with $\text{Re }\lambda_j(|\xi|)>0$, $j=1,2,3,$ for $|\xi|=\varepsilon$ and $|\xi|=N$, we complete the proof immediately.
\end{proof}

\subsection{Representations of solutions} From Lemmas \ref{Lemma 2.1} and \ref{Lemma 2.2}, we know that when $\sigma\in[0,1)\cup(1,2]$ with small frequencies or large frequencies the uniform invertibility of $T_{\sigma,\intt}$ and $T_{\sigma,\extt}$ hold. Thus, we have the next theorems for the representations of solutions. The proofs of them are based on Lemmas \ref{Lemma 2.1} and \ref{Lemma 2.2}.

\vskip2mm

\begin{thm}\label{Thm. Rep. Sol. Small freq. Fourier}
	There exists a matrix $T_{\sigma,\intt}$ for $\sigma\in[0,1)\cup(1,2]$, which is uniformly invertible for small frequencies such that the following representation formula for the Cauchy problem \eqref{Eq. first-order Fourier law} with  $\sigma\in[0,1)\cup(1,2]$ holds:
	\begin{align*}
	\chi_{\intt}(\xi)w^{(0)}(t,\xi)=\chi_{\intt}(\xi)T_{\sigma,\intt}\diag\left(e^{-\lambda_1(|\xi|)t},e^{-\lambda_2(|\xi|)t},e^{-\lambda_3(|\xi|)t}\right)T_{\sigma,\intt}^{-1}w_0^{(0)}(\xi),
	\end{align*}
	where the characteristic roots $\lambda_j(|\xi|)$ for $j=1,2,3,$ have the following asymptotic behavior:
	\begin{itemize}
		\item if $\sigma\in[0,1)$, then we have
		\begin{align*}
		\lambda_1(|\xi|)=|\xi|^{4-2\sigma},\,\,\,\,\lambda_2(|\xi|)=|\xi|^2+|\xi|^{4-2\sigma},\,\,\,\,\lambda_{3}(|\xi|)=|\xi|^{2\sigma}-2|\xi|^{4-2\sigma},
		\end{align*}
		modulo $\ml{O}\big(|\xi|^{6-4\sigma}\big)$;
		\item if $\sigma\in(1,2]$, then we have
		\begin{align*}
		\lambda_1(|\xi|)=y_1|\xi|^2,\,\,\,\,\lambda_2(|\xi|)=y_2|\xi|^2,\,\,\,\,\lambda_3(|\xi|)=y_3|\xi|^2,
		\end{align*}
		modulo $\ml{O}\big(|\xi|^{2\sigma}\big)$, where $y_1,y_2,y_3$ are determined in \eqref{Value y1 y2 y3}.
	\end{itemize}
\end{thm}

\vskip2mm

\begin{thm}\label{Thm. Rep. Sol. Large freq. Fourier}
	There exists a matrix $T_{\sigma,\extt}$ for $\sigma\in[0,1)\cup(1,2]$, which is uniformly invertible for large frequencies such that the following representation formula for the Cauchy problem \eqref{Eq. first-order Fourier law} with  $\sigma\in[0,1)\cup(1,2]$ holds:
	\begin{align*}
	\chi_{\extt}(\xi)w^{(0)}(t,\xi)=\chi_{\extt}(\xi)T_{\sigma,\extt}\diag\left(e^{-\lambda_1(|\xi|)t},e^{-\lambda_2(|\xi|)t},e^{-\lambda_3(|\xi|)t}\right)T^{-1}_{\sigma,\extt}w_0^{(0)}(\xi),
	\end{align*}
	where the characteristic roots $\lambda_j(|\xi|)$ for $j=1,2,3,$ have the following asymptotic behavior:
	\begin{itemize}
		\item if $\sigma\in[0,1)$, then we have
		\begin{align*}
		\lambda_1(|\xi|)=y_1|\xi|^2,\,\,\,\,\lambda_2(|\xi|)=y_2|\xi|^2,\,\,\,\,\lambda_3(|\xi|)=y_3|\xi|^2,
		\end{align*}
		modulo $\ml{O}\big(|\xi|^{2\sigma}\big)$, where $y_1,y_2,y_3$ are determined in \eqref{Value y1 y2 y3};
		\item if $\sigma\in(1,2]$, then we have
		\begin{align*}
		\lambda_1(|\xi|)=|\xi|^{4-2\sigma},\,\,\,\,\lambda_2(|\xi|)=|\xi|^2+|\xi|^{4-2\sigma},\,\,\,\,\lambda_{3}(|\xi|)=|\xi|^{2\sigma}-2|\xi|^{4-2\sigma},
		\end{align*}
		modulo $\ml{O}\big(|\xi|^{6-4\sigma}\big)$.
	\end{itemize}
\end{thm}

Lastly, considering \eqref{Eq. first-order Fourier law} with $\sigma=1$, from Lemma \ref{Lemma 2.3} we can derive the explicit representation of solutions in the following statement.

\vskip2mm

\begin{thm}\label{Thm. Rep. Sol. =1 freq. Fourier}
	There exists a matrix $T_{1,0}$, which is uniformly invertible for all frequencies such that the following representation formula for the Cauchy problem \eqref{Eq. first-order Fourier law} with $\sigma=1$ holds:
	\begin{align*}
	w^{(0)}(t,\xi)=T_{1,0}\diag\left(e^{-\lambda_1(|\xi|)t},e^{-\lambda_2(|\xi|)t},e^{-\lambda_3(|\xi|)t}\right)T^{-1}_{1,0}w_0^{(0)}(\xi),
	\end{align*}
	where the characteristic roots $\lambda_j(|\xi|)$ have the following explicit expressions:
	\begin{align*}
	\lambda_1(|\xi|)=y_4|\xi|^2,\,\,\,\,\lambda_2(|\xi|)=y_5|\xi|^2,\,\,\,\,\lambda_3(|\xi|)=y_6|\xi|^2,
	\end{align*}
	where $y_4,y_5,y_6$ are determined in \eqref{Value y4 y5 y6}.
\end{thm}
\section{Some qualitative properties of solutions}\label{Section Qualitative properties}

In this section we derive smoothing effect of solutions and $L^2$ well-posedness for linear thermoelastic plate equations with friction or structural damping.

Let us study smoothing effect of solutions initially. 

\vskip2mm

\begin{thm}\label{Thm smoothing effect}
	Let us assume $\left(|D|^2 u_0,u_1,\theta_0\right)\in L^2(\mb{R}^n)\times L^2(\mb{R}^n)\times L^2(\mb{R}^n).$ Then, the solution to the Cauchy problem \eqref{Eq. thermo plate Fourier} with $\sigma\in[0,2)$ belongs to Gevrey spaces such that
	\begin{align*}
	\left(|D|^{s+2}u,|D|^su_t,|D|^s\theta\right)(t,\cdot)\in\Gamma^{\kappa}(\mb{R}^n)\times \Gamma^{\kappa}(\mb{R}^n)\times\Gamma^{\kappa}(\mb{R}^n)\,\,\,\,\text{for any}\,\,\,\,t>0,
	\end{align*}
	with $s\ge0$, where the parameter $\kappa=1$ when $\sigma\in\left[0,\frac{3}{2}\right]$, and $\kappa=\frac{1}{4-2\sigma}$ when $\sigma\in\left(\frac{3}{2},2\right)$.
\end{thm}
\begin{proof}To understand Gevrey smoothing of the solution, we only need to study the regularity properties of the solution for frequencies in the large zone $Z_{\extt}(N)$. From Theorem \ref{Thm. Rep. Sol. Large freq. Fourier} we may estimate
	\begin{align*}
	\chi_{\extt}(\xi)|\xi|^s\big|w^{(0)}(t,\xi)\big|\lesssim
	\left\{
	\begin{aligned}
	&\chi_{\extt}(\xi)|\xi|^se^{-|\xi|^2t}\big|w_0^{(0)}(\xi)\big|&&\text{if}\,\,\,\,\sigma\in[0,1],\\
	&\chi_{\extt}(\xi)|\xi|^se^{-|\xi|^{4-2\sigma}t}\big|w_0^{(0)}(\xi)\big|&&\text{if}\,\,\,\,\sigma\in(1,2).
	\end{aligned}
	\right.
	\end{align*}
	When we take the parameter $\kappa$ in Gevrey spaces $\Gamma^{\kappa}(\mb{R}^n)$ such that $\kappa=1$ if $\sigma\in\left[0,\frac{3}{2}\right]$, and $\kappa=\frac{1}{4-2\sigma}$ if $\sigma\in\left(\frac{3}{2},2\right)$, they lead to
	\begin{align*}
	\ml{F}^{-1}_{\xi\rightarrow x}\left(|\xi|^sw^{(0)}(t,\xi)\right)(t,\cdot)\in \Gamma^{\kappa}(\mb{R}^n)\,\,\,\,\text{for any}\,\,\,\,t>0.
	\end{align*}
	One can complete the proof strictly following the paper \cite{Reissig2016}.
\end{proof}

\vskip2mm

\begin{rem}
	We notice that for the Cauchy problem \eqref{Eq. thermo plate Fourier} with $\sigma\in\left[0,\frac{3}{2}\right]$, the solution belongs to the Gevrey space $\Gamma^1(\mb{R}^n)$, which means analytic smoothing of the solution.
\end{rem}

\vskip2mm

\begin{rem}
	Let us consider the Cauchy problem \eqref{Eq. thermo plate Fourier} with $\sigma=2$. The representation of solutions for large frequencies from Theorem \ref{Thm. Rep. Sol. Large freq. Fourier} implies for $s\geq0$ that
	\begin{align*}
	\chi_{\extt}(\xi)|\xi|^s\big|w^{(0)}(t,\xi)\big|\lesssim \chi_{\extt}(\xi)e^{-t}|\xi|^s\big|w_0^{(0)}(\xi)\big|.
	\end{align*}
	Thus, the solution does not belong to Gevrey spaces $\Gamma^{\kappa}(\mb{R}^n)$ with $\kappa\in[1,\infty)$.
\end{rem}

\vskip2mm

\begin{rem}
	The statement of Theorem \ref{Thm smoothing effect} tells us that the thermal dissipation generated by Fourier's law has a dominant influence in comparison with friction and structural damping only if $\sigma\in\left[0,\frac{3}{2}\right]$ on Gevrey smoothing. We remark that the threshold of Gevrey smoothing of solutions is $\sigma=\frac{3}{2}$.
\end{rem}
\vskip2mm
After applying the representations of solutions from Theorems \ref{Thm. Rep. Sol. Small freq. Fourier}, \ref{Thm. Rep. Sol. Large freq. Fourier} and \ref{Thm. Rep. Sol. =1 freq. Fourier}, we immediately prove the following $L^2$ well-posedness for the Cauchy problem \eqref{Eq. thermo plate Fourier}.

\vskip2mm

\begin{thm} Let us assume $\left(|D|^2 u_0,u_1,\theta_0\right)\in L^2(\mb{R}^n)\times L^2(\mb{R}^n)\times L^2(\mb{R}^n).$ Then, there exists a uniquely determined solution to the Cauchy problem \eqref{Eq. thermo plate Fourier} with $\sigma\in[0,2]$, which satisfies 
	\begin{align*}
	u\in\ml{C}\left([0,\infty),\dot{H}^2(\mb{R}^n)\right),\quad u_t\in\ml{C}\left([0,\infty),L^2(\mb{R}^n)\right),\quad\theta\in\ml{C}\left([0,\infty),L^2(\mb{R}^n)\right).
	\end{align*}
\end{thm}

\vskip2mm

\begin{rem}
	One also can derive $H^s$ well-posedness for the Cauchy problem \eqref{Eq. thermo plate Fourier} for all $s\in\mb{R}$.
\end{rem}

\section{Estimates for solutions}\label{Subsec Fouier estimate}

This section mainly develops some estimates for solutions to linear thermoelastic plate equations with different damping mechanisms in $\mb{R}^n$, $n\geq1$. The section is organized as follows. First of all, by using phase space analysis and the representations of solutions stated in Theorems \ref{Thm. Rep. Sol. Small freq. Fourier}, \ref{Thm. Rep. Sol. Large freq. Fourier}, \ref{Thm. Rep. Sol. =1 freq. Fourier}, we derive estimates of solutions to \eqref{Eq. thermo plate Fourier} with initial data belonging to $H^s(\mb{R}^n)\cap L^m(\mb{R}^n)$ for $s\geq0$ and $m\in[1,2]$. Moreover, in spired of \cite{HorbachIkehataCharao2016,Ikehata2004,Ikehata2014}, we investigate estimates of solutions to \eqref{Eq. thermo plate Fourier} with initial data taking from wighted $L^1$ spaces, i.e., $H^s(\mb{R}^n)\cap L^{1,\delta}(\mb{R}^n)$ for $s\geq0$ and $\delta\in(0,1]$. Eventually, we study $L^p-L^q$ estimates not necessary on the conjugate line with the aid of some applications of $L^r$ estimates for oscillating integrals.

\subsection{Energy estimates}
Before stating our main results, let us denote the parameters for $n\geq1$, $s\geq0$ and $m\in[1,2]$ by the following way:
\begin{align*}
\gamma(\sigma,n,m,s):=\left\{\begin{aligned}
&\frac{(2-m)n+2ms}{4m(2-\sigma)}&&\text{if}\,\,\,\,\sigma\in[0,1),\\
&\frac{(2-m)n+2ms}{4m}&&\text{if}\,\,\,\,\sigma\in[1,2].
\end{aligned}\right.
\end{align*}
It will be used to described the decay rate of the energy estimates later.

Additionally, we define the function spaces $\ml{A}_{m,s}(\mb{R}^n)$ for $s\geq0$ and $m\in[1,2]$
\begin{align*}
\ml{A}_{m,s}(\mb{R}^n):=\left(H^s(\mb{R}^n)\cap L^m(\mb{R}^n)\right)\times \left(H^s(\mb{R}^n)\cap L^m(\mb{R}^n)\right)\times \left(H^s(\mb{R}^n)\cap L^m(\mb{R}^n)\right),
\end{align*}
and the function spaces $\ml{B}_{\delta,s}(\mb{R}^n)$ for $s\geq0$ and $\delta\in[0,1]$
\begin{align*}
\ml{B}_{\delta,s}(\mb{R}^n):=\left(H^s(\mb{R}^n)\cap L^{1,\delta}(\mb{R}^n)\right)\times \left(H^s(\mb{R}^n)\cap L^{1,\delta}(\mb{R}^n)\right)\times \left(H^s(\mb{R}^n)\cap L^{1,\delta}(\mb{R}^n)\right),
\end{align*}
carrying their corresponding norms.

\vskip2mm

\begin{thm} \label{L2-L2 esitmates Fourier law}Let us assume $\left(|D|^2u_0,u_1,\theta_0\right)\in \ml{A}_{2,s}(\mb{R}^n)$. Then, the solution to the Cauchy problem \eqref{Eq. thermo plate Fourier} with $\sigma\in[0,2]$ satisfies the following estimates:
	\begin{align*}
	\left\||D|^2u(t,\cdot)\right\|_{\dot{H}^s(\mb{R}^n)}&+\left\|u_t(t,\cdot)\right\|_{\dot{H}^s(\mb{R}^n)}+\|\theta(t,\cdot)\|_{\dot{H}^s(\mb{R}^n)}\\
	&\lesssim(1+t)^{-\gamma(\sigma,n,2,s)}\left\|\left(|D|^2u_0,u_1,\theta_0\right)\right\|_{\ml{A}_{2,s}(\mb{R}^n)}.
	\end{align*}
\end{thm}
\begin{proof}
	We can complete the proof of Theorem \ref{L2-L2 esitmates Fourier law} by applying the Parseval-Plancherel theorem.
\end{proof}

\vskip2mm

\begin{rem}
	Let us assume $\left(|D|^2u_0,u_1,\theta_0\right)\in\dot{H}^s(\mb{R}^n)\times\dot{H}^s(\mb{R}^n)\times \dot{H}^s(\mb{R}^n)$ for $s\geq0$. Then, the solution satisfies the following bounded estimates:
	\begin{align*}
	\left\||D|^2u(t,\cdot)\right\|_{\dot{H}^s(\mb{R}^n)}&+\left\|u_t(t,\cdot)\right\|_{\dot{H}^s(\mb{R}^n)}+\|\theta(t,\cdot)\|_{\dot{H}^s(\mb{R}^n)}\\
	&\lesssim\left\|\left(|D|^2u_0,u_1,\theta_0\right)\right\|_{\dot{H}^s(\mb{R}^n)\times \dot{H}^s(\mb{R}^n)\times \dot{H}^s(\mb{R}^n)}.
	\end{align*}
\end{rem}
Next, we consider initial data taking from $H^s$ with additional regularity $L^m$, $m\in[1,2)$, which implies an additional decay in the corresponding estimates.

\vskip2mm

\begin{thm} \label{Lm-L2 esitmates Fourier law}Let us assume $\left(|D|^2u_0,u_1,\theta_0\right)\in \ml{A}_{m,s}(\mb{R}^n)$, where $s\geq0$ and $m\in[1,2)$. Then, the solution to the Cauchy problem \eqref{Eq. thermo plate Fourier} with $\sigma\in[0,2]$ satisfies the next estimates:
	\begin{align*}
	\left\||D|^2u(t,\cdot)\right\|_{\dot{H}^s(\mb{R}^n)}&+\left\|u_t(t,\cdot)\right\|_{\dot{H}^s(\mb{R}^n)}+\|\theta(t,\cdot)\|_{\dot{H}^s(\mb{R}^n)}\\
	&\lesssim(1+t)^{-\gamma(\sigma,n,m,s)}\left\|\left(|D|^2u_0,u_1,\theta_0\right)\right\|_{\ml{A}_{m,s}(\mb{R}^n)}.
	\end{align*}
\end{thm}
\begin{proof}
	For frequencies in the small zone, we apply H\"older's inequality and the Hausdorff-Young inequality to get the following estimates:
	\begin{align*}
	&\left\|\chi_{\intt}(\xi)|\xi|^sw^{(0)}(t,\xi)\right\|_{L^2(\mb{R}^n)}\\
	&\qquad\quad\lesssim
	\left\{
	\begin{aligned}
	&\left\|\chi_{\intt}(\xi)|\xi|^se^{-c|\xi|^{4-2\sigma}t}\right\|_{L^{\frac{2m}{2-m}}(\mb{R}^n)}\left\|\ml{F}^{-1}\left(w^{(0)}_0\right)\right\|_{L^m(\mb{R}^n)}&&\text{if}\,\,\,\,\sigma\in[0,1),\\
	&\left\|\chi_{\intt}(\xi)|\xi|^se^{-c|\xi|^{2}t}\right\|_{L^{\frac{2m}{2-m}}(\mb{R}^n)}\left\|\ml{F}^{-1}\left(w^{(0)}_0\right)\right\|_{L^m(\mb{R}^n)}&&\text{if}\,\,\,\,\sigma\in[1,2],
	\end{aligned}
	\right.\\
	&\qquad\quad\lesssim\left\{\begin{aligned}
	&(1+t)^{-\frac{(2-m)n+2ms}{4m(2-\sigma)}}\left\|\ml{F}^{-1}\left(w^{(0)}_0\right)\right\|_{L^m(\mb{R}^n)}&&\text{if}\,\,\,\,\sigma\in[0,1),\\
	&(1+t)^{-\frac{(2-m)n+2ms}{4m}}\left\|\ml{F}^{-1}\left(w^{(0)}_0\right)\right\|_{L^m(\mb{R}^n)}&&\text{if}\,\,\,\,\sigma\in[1,2],
	\end{aligned}\right.
	\end{align*}
	where we use the following facts for $\bar{m}\in[1,\infty)$, $\alpha_0>0$ and $s\geq0$:
	\begin{align*}
	\left\|\chi_{\intt}(\xi)|\xi|^se^{-c|\xi|^{\alpha_0}t}\right\|^{\bar{m}}_{L^{\bar{m}}(\mb{R}^n)}\lesssim1\,\,\,\,\text{ if}\,\,\,\,0\leq t\leq1,
	\end{align*}
	\begin{align*}
	\left\|\chi_{\intt}(\xi)|\xi|^se^{-c|\xi|^{\alpha_0}t}\right\|_{L^{\bar{m}}(\mb{R}^n)}^{\bar{m}}&=\int_0^{\varepsilon}r^{s\bar{m}+n-1}e^{-c\bar{m}r^{\alpha_0}t}dr\\
	&=\frac{1}{\alpha_0}t^{-\frac{1}{\alpha_0}(s\bar{m}+n)}\int_0^{\varepsilon^{\alpha_0}t}\tau^{\frac{1}{\alpha_0}(s\bar{m}+n)-1}e^{-c\bar{m}\tau^{\alpha_0}}d\tau\\
	&\lesssim t^{-\frac{\bar{m}}{\alpha_0}\left(s+\frac{n}{\bar{m}}\right)}\,\,\,\,\text{ if}\,\,\,\,1\leq t.
	\end{align*}
	For frequencies in the bounded zone and the large zone, we obtain an exponential decay estimate
	\begin{align*}
	\left\|\left(\chi_{\midd}(\xi)+\chi_{\extt}(\xi)\right)|\xi|^sw^{(0)}(t,\xi)\right\|_{L^2(\mb{R}^n)}\lesssim e^{-ct}\left\|\ml{F}^{-1}\left(w^{(0)}_0\right)\right\|_{H^s(\mb{R}^n)},
	\end{align*}
	where the constant $c>0$. Finally, combining with the Parseval-Plancherel theorem, the proof of Theorem \ref{Lm-L2 esitmates Fourier law} is complete.
\end{proof}

\vskip2mm

\begin{rem}
	Concerning the sharpness of the derived energy estimates in Theorem \ref{Lm-L2 esitmates Fourier law}, we point out that the estimates for $\left\||\xi|^sw^{(0)}(t,\xi)\right\|_{L^2(\mb{R}^n)}$ seen to be sharp because diagonalization procedure is used in deriving representations of solutions.
\end{rem}
\vskip2mm
Next, we discuss energy estimates with initial data taking from the weighted spaces $L^{1,\delta}$ for $\delta\in(0,1]$ (see Notation in Section \ref{Section Introduction}).
Before stating our result, we recall the following useful lemma, which was introduced in Lemma 2.1 in the paper \cite{Ikehata2004}.

\vskip2mm

\begin{lem}\label{Ikehata lemma L1gamma repeat}
	Let $\delta\in(0,1]$ and $f\in L^{1,\delta}(\mb{R}^n)$. Then, the following estimate holds:
	\begin{align*}
	|\hat{f}(\xi)|\leq C_{\delta}|\xi|^{\delta}\|f\|_{L^{1,\delta}(\mb{R}^n)}+\Big|\int_{\mb{R}^n}f(x)dx\Big|,
	\end{align*}
	with some constant $C_{\delta}>0$.
\end{lem}

\vskip2mm

\begin{thm}\label{Energy estimates L1delta}
	Let us assume $\left(|D|^2u_0,u_1,\theta_0\right)\in\ml{B}_{\delta,s}(\mb{R}^n)$, where $s\geq0$ and $\delta\in(0,1]$. Then, the solution the Cauchy problem \eqref{Eq. thermo plate Fourier} with $\sigma\in[0,2]$ satisfies the following estimates:
	\begin{align*}
	&\left\||D|^2u(t,\cdot)\right\|_{\dot{H}^s(\mb{R}^n)}+\left\|u_t(t,\cdot)\right\|_{\dot{H}^s(\mb{R}^n)}+\|\theta(t,\cdot)\|_{\dot{H}^s(\mb{R}^n)}\\
	&\qquad\qquad\lesssim(1+t)^{-\gamma(\sigma,n,1,s+\delta)}\left\|\left(|D|^2u_0,u_1,\theta_0\right)\right\|_{\ml{B}_{\delta,s}(\mb{R}^n)}+(1+t)^{-\gamma(\sigma,n,1,s)}\Big|\int_{\mb{R}^n}U_0(x)dx\Big|,
	\end{align*}
	where initial data $U_0(x)$ is defined in \eqref{Initialdata}.
\end{thm}

\begin{proof}
	One can prove Theorem \ref{Energy estimates L1delta} strictly following Theorem 4.3 in the recent paper \cite{Chen2018}.
\end{proof}

\vskip2mm

\begin{rem}
	By restricting
	\begin{align}\label{int=0}
	\Big|\int_{\mb{R}^n}U_0(x)dx\Big|=0,
	\end{align}
	we find that the decay rate given in Theorem \ref{Lm-L2 esitmates Fourier law} when $m=1$ can be improved by $(1+t)^{-\frac{\delta}{2}}$ for $\delta\in(0,1]$. We need to point out that the additional condition \eqref{int=0} holds when $U_0(x)$ is odd function with respect to $x_n$, in other words,
	\begin{align*}
	U_0(x_1,\dots,x_{n-1},-x_n)=-U_0(x_1,\dots,x_{n-1},x_n).
	\end{align*} 
\end{rem}

\vskip2mm

\begin{rem}
	The statements of Theorems \ref{L2-L2 esitmates Fourier law}, \ref{Lm-L2 esitmates Fourier law} and \ref{Energy estimates L1delta} tell us that the thermal dissipation generated by Fourier's law has a dominant influence in comparison with friction and structural damping only if $\sigma\in[1,2]$ on energy estimates.
\end{rem}

\subsection{$L^p-L^q$ estimates not necessary on the conjugate line}
In the beginning, let us introduce the parameters to depict the decay rate
\begin{align}\label{LpLq decay rate}
\mu(\sigma,n,p,q,s):=\left\{\begin{aligned}
&\frac{s}{4-2\sigma}+\frac{n}{4-2\sigma}\left(\frac{1}{p}-\frac{1}{q}\right)&&\text{if}\,\,\,\,\sigma\in[0,1),\\
&\frac{s}{2}+\frac{n}{2}\left(\frac{1}{p}-\frac{1}{q}\right)&&\text{if}\,\,\,\,\sigma\in[1,2],
\end{aligned}\right.
\end{align}
where $s\geq0$ and $1\leq p\leq q\leq\infty$.

Moreover, we define the parameter to depict the regularity for initial data
\begin{align}\label{LpLq regularity}
M_{n,s,p,q}>s+n\left(\frac{1}{p}-\frac{1}{q}\right),
\end{align}
where $s\geq0$ and $1\leq p\leq 2\leq q\leq\infty$.

\subsubsection{$L^p-L^q$ estimates for the model with $\sigma\in[0,1)\cup(1,2]$}Before starting our main theorem, we prove the following useful lemma first.

\vskip2mm

\begin{lem}\label{Lemma LpLq}
	Let us $f\in\ml{S}(\mb{R}^n)$ and $\kappa_1>0$, $\kappa_2\geq0$, $s\geq0$. Then, the next estimates hold:
	\begin{align}
	\left\|\ml{F}^{-1}_{\xi\rightarrow x}\left(\chi_{\intt}(\xi)|\xi|^se^{-c|\xi|^{\kappa_1}t}\hat{f}(\xi)\right)\right\|_{L^q(\mb{R}^n)}&\lesssim (1+t)^{-\frac{s}{\kappa_1}-\frac{n}{\kappa_1}\left(\frac{1}{p}-\frac{1}{q}\right)}\|f\|_{L^p(\mb{R}^n)},\label{LpLq small frequencies}\\
	\left\|\ml{F}^{-1}_{\xi\rightarrow x}\left(\chi_{\extt}(\xi)|\xi|^se^{-c|\xi|^{\kappa_2}t}\hat{f}(\xi)\right)\right\|_{L^q(\mb{R}^n)}&\lesssim e^{-ct}\left\|\langle D\rangle^{M_{n,s,p,q}} f\right\|_{L^p(\mb{R}^n)},\label{LpLq large frequencies}
	\end{align}
	where $c>0$, $1\leq p\leq 2\leq q\leq\infty$ and $M_{n,s,p,q}$ is chosen in \eqref{LpLq regularity}.
\end{lem}
\begin{proof}
	Let us prove \eqref{LpLq small frequencies} first. Applying the Hausdorff-Young inequality yields
	\begin{align}\label{Proof 01}
	\left\|\ml{F}^{-1}_{\xi\rightarrow x}\left(\chi_{\intt}(\xi)|\xi|^se^{-c|\xi|^{\kappa_1}t}\hat{f}(\xi)\right)\right\|_{L^q(\mb{R}^n)}\lesssim \left\|\chi_{\intt}(\xi)|\xi|^se^{-c|\xi|^{\kappa_1}t}\hat{f}(\xi)\right\|_{L^{q'}(\mb{R}^n)}.
	\end{align}
	Here $\frac{1}{q}+\frac{1}{q'}=1$ with $2\leq q\leq \infty$. By H\"older's inequality, the estimate holds
	\begin{align}\label{Proof 02}
	\left\|\chi_{\intt}(\xi)|\xi|^se^{-c|\xi|^{\kappa_1}t}\hat{f}(\xi)\right\|_{L^{q'}(\mb{R}^n)}\lesssim\left\|\chi_{\intt}(\xi)|\xi|^se^{-c|\xi|^{\kappa_1}t}\right\|_{L^{\tilde{p}}(\mb{R}^n)}\|\hat{f}\|_{L^{p'}(\mb{R}^n)},
	\end{align}
	where $\frac{1}{q'}=\frac{1}{\tilde{p}}+\frac{1}{p'}$ with $2\leq p'\leq\infty$. Finally, combining with \eqref{Proof 01}, \eqref{Proof 02} and the Hausdorff-Young inequality leads to
	\begin{align*}
	\left\|\ml{F}^{-1}_{\xi\rightarrow x}\left(\chi_{\intt}(\xi)|\xi|^se^{-c|\xi|^{\kappa_1}t}\hat{f}(\xi)\right)\right\|_{L^q(\mb{R}^n)}
	&\lesssim(1+t)^{-\frac{s}{\kappa_1}-\frac{n}{\kappa_1}\left(\frac{1}{p}-\frac{1}{q}\right)}\|f\|_{L^p(\mb{R}^n)}.
	\end{align*}
	Next, we begin with proving \eqref{LpLq large frequencies}. For $0\leq t\leq 1$, by the similar approach we have
	\begin{align}
	&\left\|\ml{F}^{-1}_{\xi\rightarrow x}\left(\chi_{\extt}(\xi)|\xi|^se^{-c|\xi|^{\kappa_2}t}\hat{f}(\xi)\right)\right\|_{L^q(\mb{R}^n)}\nonumber\lesssim \left\|\chi_{\extt}(\xi)\langle\xi\rangle^s\hat{f}(\xi)\right\|_{L^{q'}(\mb{R}^n)}\nonumber\\
	&\qquad\qquad\lesssim\left\|\chi_{\extt}(\xi)\langle\xi\rangle^{-n\left(\frac{1}{p}-\frac{1}{q}\right)-\epsilon}\right\|_{L^{\tilde{p}}(\mb{R}^n)}\left\|\chi_{\extt}(\xi)\langle \xi\rangle^{s+n\left(\frac{1}{p}-\frac{1}{q}\right)+\epsilon}\hat{f}(\xi)\right\|_{L^{p'}(\mb{R}^n)},\label{Proof 03}
	\end{align}
	where $\frac{1}{q}+\frac{1}{q'}=1$, $\frac{1}{q'}=\frac{1}{\tilde{p}}+\frac{1}{p'}$ with $2\leq q\leq \infty$, $2\leq p'\leq\infty$ and $\epsilon>0$.\\
	The following fact holds:
	\begin{align}\label{Proof 04}
	\left\|\chi_{\extt}(\xi)\langle\xi\rangle^{-n\left(\frac{1}{p}-\frac{1}{q}\right)-\epsilon}\right\|_{L^{\tilde{p}}(\mb{R}^n)}^{\tilde{p}}=\int_N^{\infty}\langle r\rangle^{-n\left(\frac{1}{p}-\frac{1}{q}\right)\tilde{p}-\epsilon\tilde{p}+n-1}dr=\int_N^{\infty}\langle r\rangle^{-\epsilon\tilde{p}-1}dr<\infty.
	\end{align}
	Summarizing \eqref{Proof 03}, \eqref{Proof 04} and using the Hausdorff-Young inequality we derive
	\begin{align*}
	\left\|\ml{F}^{-1}_{\xi\rightarrow x}\left(\chi_{\extt}(\xi)|\xi|^se^{-c|\xi|^{\kappa_2}t}\hat{f}(\xi)\right)\right\|_{L^q(\mb{R}^n)}\lesssim \left\|\langle D\rangle^{s+n\left(\frac{1}{p}-\frac{1}{q}\right)+\epsilon}f\right\|_{L^{p}(\mb{R}^n)}
	\end{align*}
	for $1\leq p\leq 2\leq q\leq \infty$ and $0\leq t\leq 1$.\\
	For the case $t\geq1$, according to $|\xi|\geq N$ we may obtain
	\begin{align*}
	\left\|\ml{F}^{-1}_{\xi\rightarrow x}\left(\chi_{\extt}(\xi)|\xi|^se^{-c|\xi|^{\kappa_2}t}\hat{f}(\xi)\right)\right\|_{L^q(\mb{R}^n)}\lesssim e^{-ct}\left\|\langle D\rangle^{s+n\left(\frac{1}{p}-\frac{1}{q}\right)+\epsilon}f\right\|_{L^{p}(\mb{R}^n)}.
	\end{align*}
	Hence, the proof of Lemma \ref{Lemma LpLq} is completed.
\end{proof}

Now, let us derive $L^p-L^q$ estimates of solutions to the Cauchy problem \eqref{Eq. thermo plate Fourier} with $\sigma\in[0,1)\cup(1,2]$, where $1\leq p\leq 2\leq q\leq\infty$.

\vskip2mm
\begin{thm} \label{Lp-Lq esitmates Fourier law} Let us assume $\left(|D|^2u_0,u_1,\theta_0\right)\in\ml{S}(\mb{R}^n)\times\ml{S}(\mb{R}^n)\times\ml{S}(\mb{R}^n).$ Then, the solution to  the Cauchy problem \eqref{Eq. thermo plate Fourier} with $\sigma\in[0,1)
	\cup(1,2]$ satisfies the following estimates:
	\begin{align*}
	&\left\||D|^2u(t,\cdot)\right\|_{\dot{H}^s_q(\mb{R}^n)}+\|u_t(t,\cdot)\|_{\dot{H}^s_q(\mb{R}^n)}+\|\theta(t,\cdot)\|_{\dot{H}^s_q(\mb{R}^n)}\\
	&\qquad\lesssim(1+t)^{-\mu(\sigma,n,p,q,s)}\left\|\left(|D|^2u_0,u_1,\theta_0\right)\right\|_{H_p^{M_{n,s,p,q}}(\mb{R}^n)\times H_p^{M_{n,s,p,q}}(\mb{R}^n)\times H_p^{M_{n,s,p,q}}(\mb{R}^n)},
	\end{align*}
	with $s\geq0$, $1\leq p\leq 2\leq q\leq\infty$ and $M_{n,s,p,q}>s+n\left(\frac{1}{p}-\frac{1}{q}\right)$.
\end{thm}
\vskip2mm
\begin{rem}
	If one is interested in the case $p\in(1,2]$, then we can choose $M_{n,s,p,q}=s+n\left(\frac{1}{p}-\frac{1}{q}\right)$.
\end{rem}
\begin{proof}
	From Theorems \ref{Thm. Rep. Sol. Small freq. Fourier}, \ref{Thm. Rep. Sol. =1 freq. Fourier}, \ref{Thm. Rep. Sol. =1 freq. Fourier} and Lemma \ref{Lemma 2.4} we obtain
	\begin{align*}
	&\left\||D|^s\ml{F}^{-1}_{\xi\rightarrow x}\left(w^{(0)}\right)(t,\cdot)\right\|_{L^q(\mb{R}^n)}	\lesssim \left\||\xi|^sw^{(0)}(t,\xi)\right\|_{L^{q'}(\mb{R}^n)}\\
	&\qquad\qquad\lesssim\left\|\chi_{\intt}(\xi)|\xi|^sw^{(0)}(t,\xi)\right\|_{L^{q'}(\mb{R}^n)}+\left\|\chi_{\midd}(\xi)|\xi|^sw^{(0)}(t,\xi)\right\|_{L^{q'}(\mb{R}^n)}\\
	&\qquad\qquad\quad\,\,+\left\|\chi_{\extt}(\xi)|\xi|^sw^{(0)}(t,\xi)\right\|_{L^{q'}(\mb{R}^n)},
	\end{align*}
	where $\frac{1}{q}+\frac{1}{q'}=1$ with $2\leq q\leq\infty$.\\
	Following all steps from Lemma \ref{Lemma LpLq} we immediately complete the proof.
\end{proof}

\subsubsection{$L^p-L^q$ estimates for the model with $\sigma=1$}Due to the treatment in Lemma \ref{Lemma 2.3}, it allows us to obtain explicit representation of solutions. Therefore, it is helpful for us to derive $L^p-L^q$ estimates of solutions to the Cauchy problem \eqref{Eq. thermo plate Fourier}, where $1\leq p\leq q\leq\infty$. To do this, let us introduce some results in $L^p$ estimates for some oscillating integral by using modified Bessel functions (c.f. \cite{ReissigEbert2018,DuongKainaneReissig2015}).

\vskip2mm
\begin{lem}\label{Lemma LpLq away}
	Let $p\in[1,\infty]$ and $c_1>0$, $c_2\neq0$. Then, the following estimates hold for any $t>0$:
	\begin{align}
	\left\|\ml{F}^{-1}_{\xi\rightarrow x}\left(|\xi|^se^{-c_1|\xi|^2t}\sin\left(c_2|\xi|^2t\right)\right)(t,\cdot)\right\|_{L^p(\mb{R}^n)}\lesssim t^{-\frac{s}{2}-\frac{n}{2}\left(1-\frac{1}{p}\right)},\label{Proof 06}\\
	\left\|\ml{F}^{-1}_{\xi\rightarrow x}\left(|\xi|^se^{-c_1|\xi|^2t}\cos\left(c_2|\xi|^2t\right)\right)(t,\cdot)\right\|_{L^p(\mb{R}^n)}\lesssim t^{-\frac{s}{2}-\frac{n}{2}\left(1-\frac{1}{p}\right)},\label{Proof 07}
	\end{align}
	where $s\geq0$ and $n\geq1$.
\end{lem}
\begin{proof}
	For the proof of \eqref{Proof 07}, one can see Proposition 12 in \cite{DuongKainaneReissig2015}. One can prove \eqref{Proof 06} by some minor modifications of the proof of Proposition 12 in \cite{DuongKainaneReissig2015}.
\end{proof}

\vskip2mm
\begin{thm}\label{Thm, Lp-Lq away}
	Let us assume $\left(|D|^2u_0,u_1,\theta_0\right)\in L^p(\mb{R}^n)\times L^p(\mb{R}^n)\times L^p(\mb{R}^n)$, where $p\geq1$. Then, the solution to the Cauchy problem \eqref{Eq. thermo plate Fourier} with $\sigma=1$ satisfies the next estimates:
	\begin{align*}
	&\left\||D|^2u(t,\cdot)\right\|_{\dot{H}^s_q(\mb{R}^n)}+\|u_t(t,\cdot)\|_{\dot{H}^s_q(\mb{R}^n)}+\|\theta(t,\cdot)\|_{\dot{H}^s_q(\mb{R}^n)}\\
	&\qquad\qquad\lesssim t^{-\mu(1,n,p,q,s)}   \left\|\left(|D|^2u_0,u_1,\theta_0\right)\right\|_{L^p(\mb{R}^n)\times L^p(\mb{R}^n)\times L^p(\mb{R}^n)},
	\end{align*}
	where $s\geq0$ and $1\leq p\leq q\leq\infty$.
\end{thm}
\begin{proof}
	From Theorem \ref{Thm. Rep. Sol. =1 freq. Fourier}, the solutions to \eqref{Eq. thermo plate Fourier} can be explicitly represented by the following way:
	\begin{align}
	&\left(u_t+|D|^2u,u_t-|D|^2u,\theta\right)^{\mathrm{T}}(t,x)\nonumber\\
	&\qquad\qquad=\left(\sum\limits_{j,k=1}^3c_{jkl}\ml{F}^{-1}_{\xi\rightarrow x}\left(e^{-\text{Re }y_{j+3}|\xi|^2t-i\text{Im }y_{j+3}|\xi|^2t}\right)\ast_{(x)}U_{0,k}(x)\right)_{l=1}^3\nonumber\\
	&\qquad\qquad=\left(\sum\limits_{j,k=1}^3c_{jkl}\left(K_0^{(j)}(t,x)+K_1^{(j)}(t,x)\right)\ast_{(x)}U_{0,k}(x)\right)_{l=1}^3,\label{Another repre}
	\end{align}
	where $c_{jkl}$ are constants and the kernels are
	\begin{align}
	K_0^{(j)}&:=\ml{F}_{\xi\rightarrow x}^{-1}\left(-i\sin\left(\text{Im }y_{j+3}|\xi|^2t\right)e^{-\text{Re }y_{j+3}|\xi|^2t}\right),\label{Kernel 00}\\
	K_1^{(j)}&:=\ml{F}_{\xi\rightarrow x}^{-1}\left(\cos\left(\text{Im }y_{j+3}|\xi|^2t\right)e^{-\text{Re }y_{j+3}|\xi|^2t}\right)\label{Kernel 01}.
	\end{align}
	By applying Lemma \ref{Lemma LpLq away} we get
	\begin{align*}
	\sum\limits_{j=1}^3\left\||D|^sK_0^{(j)}(t,\cdot)\right\|_{L^r(\mb{R}^n)}+\sum\limits_{j=1}^3\left\||D|^sK_1^{(j)}(t,\cdot)\right\|_{L^r(\mb{R}^n)}\lesssim t^{-\frac{s}{2}-\frac{n}{2}\left(1-\frac{1}{r}\right)}
	\end{align*}
	for all $r\in[1,\infty]$. Then, we directly apply Young's inequality in \eqref{Another repre} to complete the proof.
\end{proof}

We find that if $p=q$ in Theorem \ref{Thm, Lp-Lq away} and we suppose that higher regularity for initial data, the singularity as $t\rightarrow+0$ will disappear. So, we have the next result.

\vskip2mm
\begin{coro}\label{Coro LpLq}
	Let us assume $\left(|D|^2u_0,u_1,\theta_0\right)\in \dot{H}^s_p(\mb{R}^n)\times \dot{H}^s_p(\mb{R}^n)\times \dot{H}^s_p(\mb{R}^n)$, where $p\geq1$ and $s\geq0$. Then, the solution to the Cauchy problem \eqref{Eq. thermo plate Fourier} with $\sigma=1$ satisfies the following bounded estimates:
	\begin{align*}
	\left\||D|^2u(t,\cdot)\right\|_{\dot{H}^s_p(\mb{R}^n)}&+\|u_t(t,\cdot)\|_{\dot{H}^s_p(\mb{R}^n)}+\|\theta(t,\cdot)\|_{\dot{H}^s_p(\mb{R}^n)}\\
	&\lesssim   \left\|\left(|D|^2u_0,u_1,\theta_0\right)\right\|_{\dot{H}^s_p(\mb{R}^n)\times \dot{H}^s_p(\mb{R}^n)\times \dot{H}^s_p(\mb{R}^n)}.
	\end{align*}
\end{coro}

\vskip2mm
\begin{rem}
	One can apply Theorem \ref{Thm, Lp-Lq away} for $t> t_0\gg1$ and Corollary \ref{Coro LpLq} for $0\leq t\leq t_0$ to obtain decay estimates for
	\begin{align*}
	\left\||D|^2u(t,\cdot)\right\|_{\dot{H}^s_q(\mb{R}^n)}+\|u_t(t,\cdot)\|_{\dot{H}^s_q(\mb{R}^n)}+\|\theta(t,\cdot)\|_{\dot{H}^s_q(\mb{R}^n)}
	\end{align*}
	with decay rate $(1+t)^{-\mu(1,n,p,q,s)}$. At this time, initial data should belong to the function spaces $\dot{H}^s_q(\mb{R}^n)\cap L^p(\mb{R}^n)$, where $n\geq1$, $1\leq p\leq q\leq\infty$ and $s\geq0$.
\end{rem}
\vskip2mm
\begin{rem}
	The statements of Theorems \ref{Lp-Lq esitmates Fourier law} and \ref{Thm, Lp-Lq away} indicate that the thermal dissipation generated by Fourier's law has a dominant influence in comparison with friction and structural damping only if $\sigma\in[1,2]$ on $L^p-L^q$ estimates away the conjugate line.
\end{rem}

\section{Diffusion phenomena}\label{Subsec Fourier asym profile}

It is well known that diffusion phenomena allow one to bridge the decay behavior of the solution to \eqref{Eq. thermo plate Fourier} with the solution for the corresponding evolution. It also provides a tool to tackle the asymptotic profiles of solutions.  In this section we study diffusion phenomena of solutions to the Cauchy problem \eqref{Eq. thermo plate Fourier} for $\sigma\in[0,1)\cup(1,2]$ with initial data carrying different assumptions on the regularity.

In the view of the derived estimates of solutions in Section \ref{Subsec Fouier estimate}, we find that the decay rate of estimates of solutions are determined by the behavior of the characteristic roots for $\xi\in Z_{\intt}(\varepsilon)$ only. For $\xi\in Z_{\midd}(\varepsilon,N)\cup Z_{\extt}(N)$, the solutions satisfies an exponential decay when we assume initial data having suitable regularities. For this reason, we explain diffusion phenomena of solutions for small frequencies in this section.

\vskip2mm
\begin{rem}
	Considering $\sigma=1$ in the system \eqref{Eq. first-order Fourier law}, we find that $e^{-y_j|\xi|^2t}$ with $y_j\in\mb{C}$ for $j=4,5,6,$ plays a determined role in the explicit representation of $w^{(0)}(t,\xi)$ from Theorem \ref{Thm. Rep. Sol. =1 freq. Fourier}. Then, there is not any improvement in the decay estimates for the difference between the solutions to the system \eqref{Eq. first-order Fourier law} with $\sigma=1$ and the solutions to its reference system. Hence, we  explain diffusion phenomena for $\sigma\in[0,1)\cup(1,2]$ only.
\end{rem}

\subsection{Diffusion phenomena for the model with $\sigma\in[0,1)$}
To describe diffusion phenomena of the solutions to the Cauchy problem \eqref{Eq. first-order Fourier law} with $\sigma\in[0,1)$, we consider the following reference system:
\begin{align}\label{evolution Fourier law}
\left\{\begin{aligned}
&\tilde{u}_t+\diag\left(\big(-\Delta\big)^{2-\sigma},\big(-\Delta\big),\big(-\Delta\big)^{\sigma}\right)\tilde{u}=0,&&t>0,\,\,x\in\mb{R}^n,\\
&\tilde{u}(0,x)=\ml{F}^{-1}\left(T_{1\frac{1}{2}}^{-1}T_{1}^{-1}T_0^{-1}w_0^{(0)}(\xi)\right)(x),&&x\in\mb{R}^n,
\end{aligned}\right.
\end{align}
where $\tilde{u}=\left(\tilde{u}^{(1)},\tilde{u}^{(2)},\tilde{u}^{(3)}\right)^{\mathrm{T}}$ and $T_0,T_1,T_{1\frac{1}{2}}$ are defined in Lemma \ref{Lemma 2.1}.  By applying the partial Fourier transform $\tilde{w}(t,\xi)=\ml{F}_{x\rightarrow\xi}\left(\tilde{u}(t,x)\right)$, \eqref{evolution Fourier law} can be transformed to
\begin{align}\label{evolution Four. Fourier law}
\left\{\begin{aligned}
&\tilde{w}_t+\diag\left(|\xi|^{4-2\sigma},|\xi|^{2},|\xi|^{2\sigma}\right)\tilde{w}=0,&&t>0,\,\,\xi\in\mb{R}^n,\\
&\tilde{w}(0,\xi)=T_{1\frac{1}{2}}^{-1}T_{1}^{-1}T_0^{-1}w_0^{(0)}(\xi),&&\xi\in\mb{R}^n.
\end{aligned}\right.
\end{align}
We know that the solution $\tilde{w}=\tilde{w}(t,\xi)$ to \eqref{evolution Four. Fourier law} can be explicitly represented by
\begin{align}\label{evolution solution Fourier law}
\tilde{w}(t,\xi)=\diag\left(e^{-|\xi|^{4-2\sigma}t},e^{-|\xi|^{2}t},e^{-|\xi|^{2\sigma}t}\right)T_{1\frac{1}{2}}^{-1}T_{1}^{-1}T_0^{-1}w_0^{(0)}(\xi).
\end{align}

\vskip2mm
\begin{rem}
	According to the evolution system \eqref{evolution Fourier law} with $\sigma=0$, we find that the reference system is consisted of two different evolution equations such that
	\begin{align*}
	\begin{aligned}
	&\text{fourth-order parabolic equation equation:}&&\tilde{u}^{(1)}_{t}+\Delta^2\tilde{u}^{(1)}=0,\\
	&\text{heat equation:}&&\tilde{u}^{(2)}_t-\Delta\tilde{u}^{(2)}=0.
	\end{aligned}
	\end{align*}
	Therefore, we obtain double diffusion phenomena of solution to \eqref{Eq. thermo plate Fourier} with $\sigma=0$. The effect of double diffusion phenomena was introduced in the recent papers \cite{DabbiccoEbert2014,ChenReissig2018}.
\end{rem}

\vskip2mm
\begin{rem} Let us consider \eqref{Eq. thermo plate Fourier} with $\sigma\in(0,1)$. Inspiring from the dominant asymptotic behavior of eigenvalues such that \begin{align*}
	\lambda_1(|\xi|)=\ml{O}\big(|\xi|^{4-2\sigma}\big),\,\,\,\, \lambda_2(|\xi|)=\ml{O}\big(|\xi|^2\big),\,\,\,\, \lambda_3(|\xi|)=\ml{O}\big(|\xi|^{2\sigma}\big)
	\end{align*} for $\xi\in Z_{\intt}(\varepsilon)$, we observe that the evolution system \eqref{evolution Fourier law} is consisted of three different evolution equations, which are
	\begin{align*}
	\begin{aligned}
	&\text{fractional heat equation 1:}&&\tilde{u}^{(1)}_t+(-\Delta)^{2-\sigma}\tilde{u}^{(1)}=0,\\
	&\text{heat equation:}&&\tilde{u}^{(2)}_{t}-\Delta\tilde{u}^{(2)}=0,\\
	&\text{fractional heat equation 2:}&&\tilde{u}^{(3)}_t+(-\Delta)^{\sigma}\tilde{u}^{(3)}=0.
	\end{aligned}
	\end{align*}
	We may interpret this effect as triple diffusion phenomena, which is a nature generalization of the effect of double diffusion phenomena.
\end{rem}

\vskip2mm
\begin{thm}\label{Thm Diffusion Lm}
	Let us consider the Cauchy problem \eqref{Eq. first-order Fourier law} with $\sigma\in[0,1)$. We assume $\left(|D|^2u_0,u_1,\theta_0\right)\in L^m(\mb{R}^n)\times L^m(\mb{R}^n)\times L^m(\mb{R}^n)$ with $m\in[1,2]$. Then, we have the following refinement estimates:
	\begin{align*}
	&\left\|\chi_{\intt}(D)\ml{F}_{\xi\rightarrow x}^{-1}\left(w^{(0)}-T_0T_1T_{1\frac{1}{2}}\tilde{w}\right)(t,\cdot)\right\|_{\dot{H}^s(\mb{R}^n)}\\
	&\qquad\qquad\lesssim(1+t)^{-\frac{(2-m)n+2ms}{4m(2-\sigma)}-\frac{1-\sigma}{2-\sigma}}\left\|\left(|D|^2u_0,u_1,\theta_0\right)\right\|_{L^m(\mb{R}^n)\times L^m(\mb{R}^n)\times L^m(\mb{R}^n)},
	\end{align*}
	where $T_0,T_1,T_{1\frac{1}{2}}$ are defined in Lemma \ref{Lemma 2.1}.
\end{thm}
\begin{proof}
	According to the representations of solutions for small frequencies in Theorem \ref{Thm. Rep. Sol. Small freq. Fourier} and the definition of the matrices in \eqref{T matrices Case 2.1}, we may obtain
	\begin{align*}
	\chi_{\intt}(\xi)|\xi|^s\left(w^{(0)}-T_0T_1T_{1\frac{1}{2}}\tilde{w}\right)(t,\xi)=\chi_{\intt}(\xi)|\xi|^s\left(J_1(t,|\xi|)+J_2(t,|\xi|)+J_3(t,|\xi|)\right),
	\end{align*}
	where
	\begin{align*}
	J_0(t,|\xi|)&=\diag\left(e^{-\lambda_1(|\xi|)t}-e^{-|\xi|^{4-2\sigma}t},e^{-\lambda_2(|\xi|)t}-e^{-|\xi|^2t},e^{-\lambda_3(|\xi|)t}-e^{-|\xi|^{2\sigma}t}\right),\\
	J_1(t,|\xi|)&=T_0T_1T_{1\frac{1}{2}}J_0(t,|\xi|)T_{1\frac{1}{2}}^{-1}T_1^{-1}T_0^{-1}w_0^{(0)}(\xi),\\
	J_2(t,|\xi|)&=T_0T_1T_{1\frac{1}{2}}N_2(|\xi|)\diag\left(e^{-\lambda_1(|\xi|)t},e^{-\lambda_2(|\xi|)t},e^{-\lambda_3(|\xi|)t}\right)T_2^{-1}T_1^{-1}T_0^{-1}w_0^{(0)}(\xi),\\
	J_3(t,|\xi|)&=-T_0T_1T_{1\frac{1}{2}}T_2\diag\left(e^{-\lambda_1(|\xi|)t},e^{-\lambda_2(|\xi|)t},e^{-\lambda_3(|\xi|)t}\right)T_2^{-1}N_2(|\xi|)T_1^{-1}T_0^{-1}w_0^{(0)}(\xi),
	\end{align*}
	with $N_2(|\xi|)=\ml{O}\big(|\xi|^{2-2\sigma}\big)$ for $\sigma\in[0,1)$.
	
	Let us define
	\begin{align*}
	\begin{aligned}
	h_1(|\xi|)&=|\xi|^{4-2\sigma},&&h_2(|\xi|)=|\xi|^{2},&&h_3(|\xi|)=|\xi|^{2\sigma},\\
	g_1(|\xi|)&=\lambda_1(|\xi|)-h_1(|\xi|),&&g_2(|\xi|)=\lambda_2(|\xi|)-h_2(|\xi|),&&g_3(|\xi|)=\lambda_3(|\xi|)-h_3(|\xi|).
	\end{aligned}
	\end{align*}
	Applying the following formula for $j=1,2,3$:
	\begin{align*}
	e^{-h_j(|\xi|)t+g_j(|\xi|)t}-e^{-h_j(|\xi|)t}=-g_j(|\xi|)te^{-h_j(|\xi|)t}\int_0^1e^{-g_j(|\xi|)t\tau}d\tau,
	\end{align*}
	we can get
	\begin{align*}
	&\left\|\chi_{\intt}(\xi)|\xi|^s\left(w^{(0)}-T_0T_1T_{1\frac{1}{2}}\tilde{w}\right)(t,\xi)\right\|_{L^2(\mb{R}^n)}\\
	&\qquad\qquad=\left\|\chi_{\intt}(\xi)|\xi|^s\left(J_1(t,|\xi|)+J_2(t,|\xi|)+J_3(t,|\xi|)\right)\right\|_{L^2(\mb{R}^n)}\\
	&\qquad\qquad\lesssim\left\|\chi_{\intt}(\xi)|\xi|^{s+2-2\sigma}e^{-|\xi|^{4-2\sigma}t}\right\|_{L^{\frac{2m}{2-m}}(\mb{R}^n)}\left\|\ml{F}^{-1}\left(w_0^{(0)}\right)\right\|_{L^m(\mb{R}^n)}\\
	&\qquad\qquad\lesssim(1+t)^{-\gamma(\sigma,n,m,s)-\frac{1-\sigma}{2-\sigma}}\left\|\ml{F}^{-1}\left(w_0^{(0)}\right)\right\|_{L^m(\mb{R}^n)}.
	\end{align*}
	Thus, the proof is complete.
\end{proof}

\vskip2mm
\begin{thm}\label{Thm Diffusion Lp}
	Let us consider the Cauchy problem \eqref{Eq. first-order Fourier law} with $\sigma\in[0,1)$. We assume $\left(|D|^2u_0,u_1,\theta_0\right)\in \ml{S}(\mb{R}^n)\times \ml{S}(\mb{R}^n)\times \ml{S}(\mb{R}^n).$ Then, we have the following refinement estimates:
	\begin{align*}
	&\left\|\chi_{\intt}(D)\ml{F}_{\xi\rightarrow x}^{-1}\left(w^{(0)}-T_0T_1T_{1\frac{1}{2}}\tilde{w}\right)(t,\cdot)\right\|_{\dot{H}^s_q(\mb{R}^n)}\\
	&\qquad\qquad\lesssim(1+t)^{-\frac{s}{4-2\sigma}-\frac{n}{4-2\sigma}\left(\frac{1}{p}-\frac{1}{q}\right)-\frac{1-\sigma}{2-\sigma}}\left\|\left(|D|^2u_0,u_1,\theta_0\right)\right\|_{L^p(\mb{R}^n)\times L^{p}(\mb{R}^n)\times L^p(\mb{R}^n)},
	\end{align*}
	where $T_0,T_1,T_{1\frac{1}{2}}$ are defined in Lemma \ref{Lemma 2.1} with $s\geq0$, $1\leq p\leq 2\leq q\leq\infty$.
\end{thm}
\begin{proof}
	We may prove this result immediately by using Lemma \ref{Lemma LpLq}.
\end{proof}

\vskip2mm
\begin{rem}
	From the statements of Theorems \ref{Thm Diffusion Lm} and \ref{Thm Diffusion Lp}, we know when $\sigma\in[0,1)$ the thermal dissipation generated by Fourier's law and friction or structural damping have the influence on the reference system at the same time. However, the friction and structural damping have a dominant influence on the decay rate of the estimates.
\end{rem}

\subsection{Diffusion phenomena for the model with $\sigma\in(1,2]$}
Now, we describe diffusion phenomena of the solutions to the Cauchy problem \eqref{Eq. first-order Fourier law} with $\sigma\in(1,2]$ by the reference system as follows:
\begin{align}\label{evolution Fourier law 2}
\left\{\begin{aligned}
&\tilde{u}_t-\diag(y_1,y_2,y_3)\Delta\tilde{u}=0,&&t>0,\,\,x\in\mb{R}^n,\\
&\tilde{u}(0,x)=\ml{F}^{-1}\left(T_0^{-1}w_0^{(0)}(\xi)\right)(x),&&x\in\mb{R}^n,
\end{aligned}\right.
\end{align}
where $y_1,y_2,y_3\in\mb{C}$ are determined in \eqref{Value y1 y2 y3} and $T_0$ are defined in Lemma \ref{Lemma 2.2}. Applying the partial Fourier transform $\tilde{w}(t,\xi)=\ml{F}_{x\rightarrow\xi}\left(\tilde{u}(t,x)\right)$ implies
\begin{align}\label{evolution Four. Fourier law 2}
\left\{\begin{aligned}
&\tilde{w}_t+\diag(y_1,y_2,y_3)|\xi|^2\tilde{w}=0,&&t>0,\,\,\xi\in\mb{R}^n,\\
&\tilde{w}(0,\xi)=T_0^{-1}w_0^{(0)}(\xi),&&\xi\in\mb{R}^n.
\end{aligned}\right.
\end{align}
The solution to \eqref{evolution Four. Fourier law 2} is explicitly given by
\begin{align}\label{evolution solution Fourier law 2}
\tilde{w}(t,\xi)=\diag\left(e^{-y_1|\xi|^{2}t},e^{-y_2|\xi|^{2}t},e^{-y_3|\xi|^{2}t}\right)T_0^{-1}w_0^{(0)}(\xi).
\end{align}

\vskip2mm
\begin{rem}
	From \eqref{evolution Fourier law 2}, the reference system is consisted of evolution equations
	\begin{align*}
	\tilde{u}^{(j)}_t-\text{Re }y_j\Delta \tilde{u}^{(j)}-i\text{Im }y_j\Delta \tilde{u}^{(j)}=0,
	\end{align*}
	for $j=1,2,3.$ Here we interpret this effect as classical diffusion phenomenon.  Moreover, we have to point out that the reference system \eqref{evolution Fourier law 2} is consisted of heat systems and Schr\"odinger systems due to the fact that $\text{Re } y_j>0$ and $\text{Im }y_j\neq0$ for all $j=1,2,3$.
\end{rem}

\vskip2mm
\begin{rem}
	If one considers the reference system as the following heat system only:
	\begin{align}\label{Sup refer 01}
	\tilde{u}_t-\diag(\text{Re }y_1,\text{Re }y_2,\text{Re }y_3)\Delta\tilde{u}=0,
	\end{align}
	or the following Schr\"odinger system only:
	\begin{align}\label{Sup refer 02}
	\tilde{u}_t-i\diag(\text{Im }y_1,\text{Im }y_2,\text{Im }y_3)\Delta\tilde{u}=0,
	\end{align}
	we cannot observe any diffusion structure for $\sigma\in(1,2]$. In other words, comparing with Theorems \ref{Thm Diffusion Lm} and \ref{Thm Diffusion Lp}, respectively, we observe that there is not any improvement in the decay estimates for the difference between the solutions to the system \eqref{Eq. first-order Fourier law} with $\sigma\in(1,2]$ and the solutions to the reference systems \eqref{Sup refer 01} or \eqref{Sup refer 02}.
\end{rem}
\vskip2mm
Similar as last subsection, one can prove the following results.

\vskip2mm
\begin{thm}\label{Thm Diffusion Lm 1}
	Let us consider the Cauchy problem \eqref{Eq. first-order Fourier law} with $\sigma\in(1,2]$. We assume $\left(|D|^2u_0,u_1,\theta_0\right)\in L^m(\mb{R}^n)\times L^m(\mb{R}^n)\times L^m(\mb{R}^n)$ with $m\in[1,2]$. Then, we have the following refinement estimates:
	\begin{align*}
	&\left\|\chi_{\intt}(D)\ml{F}_{\xi\rightarrow x}^{-1}\left(w^{(0)}-T_0\tilde{w}\right)(t,\cdot)\right\|_{\dot{H}^s(\mb{R}^n)}\\
	&\qquad\qquad\lesssim(1+t)^{-\frac{(2-m)n+2ms}{4m}-(\sigma-1)}\left\|\left(|D|^2u_0,u_1,\theta_0\right)\right\|_{L^m(\mb{R}^n)\times L^m(\mb{R}^n)\times L^m(\mb{R}^n)},
	\end{align*}
	where $T_0$ are defined in Lemma \ref{Lemma 2.2}.
\end{thm}

\vskip2mm
\begin{thm}\label{Thm Diffusion Lp 2}
	Let us consider the Cauchy problem \eqref{Eq. first-order Fourier law} with $\sigma\in(1,2]$. We assume $\left(|D|^2u_0,u_1,\theta_0\right)\in \ml{S}(\mb{R}^n)\times \ml{S}(\mb{R}^n)\times \ml{S}(\mb{R}^n).$ Then, we have the next refinement estimates:
	\begin{align*}
	&\left\|\chi_{\intt}(D)\ml{F}_{\xi\rightarrow x}^{-1}\left(w^{(0)}-T_0\tilde{w}\right)(t,\cdot)\right\|_{\dot{H}^s_q(\mb{R}^n)}\\
	&\qquad\qquad\lesssim(1+t)^{-\frac{s}{2}-\frac{n}{2}\left(\frac{1}{p}-\frac{1}{q}\right)-(\sigma-1)}\left\|\left(|D|^2u_0,u_1,\theta_0\right)\right\|_{L^p(\mb{R}^n)\times L^{p}(\mb{R}^n)\times L^p(\mb{R}^n)},
	\end{align*}
	where $T_0$ are defined in Lemma \ref{Lemma 2.2} with $s\geq0$, $1\leq p\leq 2\leq q\leq\infty$.
\end{thm}

\vskip2mm
\begin{rem}
	From Theorems \ref{Thm Diffusion Lm} \ref{Thm Diffusion Lp}, \ref{Thm Diffusion Lm 1} and \ref{Thm Diffusion Lp 2}, the diffusion structure appears for the Cauchy problem \eqref{Eq. thermo plate Fourier} with $\sigma\in[0,1)\cup(1,2]$. More precisely, comparing Theorems \ref{Lm-L2 esitmates Fourier law} and \ref{Lp-Lq esitmates Fourier law} with Theorems \ref{Thm Diffusion Lm} and \ref{Thm Diffusion Lp}, respectively, we observe that the decay rate can be improved by $-\frac{1-\sigma}{2-\sigma}$ if $\sigma\in[0,1)$ as $t\rightarrow \infty$. In addition, comparing Theorems \ref{Lm-L2 esitmates Fourier law} and \ref{Lp-Lq esitmates Fourier law} with Theorems \ref{Thm Diffusion Lm 1} and \ref{Thm Diffusion Lp 2}, respectively, we observe that the decay rate can be improved by $-(\sigma-1)$ if $\sigma\in(1,2]$ as $t\rightarrow \infty$.
\end{rem}

\vskip2mm
\begin{rem}
	According to Theorems \ref{Thm Diffusion Lm 1} and \ref{Thm Diffusion Lp 2}, the thermal dissipation generated by Fourier's law has a dominant influence in comparison with structural damping on diffusion phenomena when $\sigma\in(1,2]$.
\end{rem}

\section{Asymptotic profiles of solutions}\label{Section Asymptotic behavior}

Our main purpose in this section is to give asymptotic profiles of solutions to the Cauchy problem \eqref{Eq. thermo plate Fourier} in a framework of the weighted $L^1$ data. The idea is motivated by \cite{Ikehata2013,Ikehata2014}.

In Section \ref{Subsec Fouier estimate} we derived the following estimates for upper bounds of solutions with weighted $L^1$ data:
\begin{align*}
\left\|U(t,\cdot)\right\|_{\dot{H}^s(\mb{R}^n)}\lesssim&(1+t)^{-\gamma(\sigma,n,1,s+1)}\left\|U_0\right\|_{H^s(\mb{R}^n)\cap L^{1,1}(\mb{R}^n)}+(1+t)^{-\gamma(\sigma,n,1,s)}\Big|\int_{\mb{R}^n}U_0(x)dx\Big|,
\end{align*}
where $\sigma\in[0,2]$, $n\geq1$, $s\geq0$. Here the solution $U(t,x)$ and data $U_0(x)$ are defined in \eqref{Solution} and \eqref{Initialdata}, respectively.

The natural questions are as follows. What is the estimate for the lower bounds of $\|U(t,\cdot)\|_{\dot{H}^s(\mb{R}^n)}$ in a framework of weighted $L^1$ data? Does this estimate is sharp? To answer these questions, we show some useful lemmas initially. Here Lemmas \ref{Pre 01} and \ref{Pre 02} have been proved in the papers \cite{HorbachIkehataCharao2016,Ikehata2004}.

\vskip2mm
\begin{lem}\label{Pre 01}
	Let $f\in L^1(\mb{R}^n)$. Then, we can expand $\hat{f}(\xi)$ by
	\begin{align*}
	\hat{f}(\xi)=A_f(\xi)-iB_f(\xi)+P_f\,\,\,\,\text{for all}\,\,\,\,\xi\in\mb{R}^n,
	\end{align*}
	where
	\begin{align*}
	A_f(\xi)&:=(2\pi)^{-\frac{n}{2}}\int_{\mb{R}^n}(\cos(x\cdot\xi)-1)f(x)dx,\\
	B_f(\xi)&:=(2\pi)^{-\frac{n}{2}}\int_{\mb{R}^n}\sin(x\cdot\xi)f(x)dx,\\
	P_f&:=(2\pi)^{-\frac{n}{2}}\int_{\mb{R}^n}f(x)dx.
	\end{align*}
\end{lem}

\vskip2mm
\begin{lem}\label{Pre 02}
	Let us consider $A_f(\xi)$ and $B_f(\xi)$ defined in Lemma \ref{Pre 01}. Then, we have the following estimates for them:
	\begin{align*}
	|A_f(\xi)|&\lesssim |\xi|\|f\|_{L^{1,1}(\mb{R}^n)},\\
	|B_f(\xi)|&\lesssim |\xi|\|f\|_{L^{1,1}(\mb{R}^n)}.
	\end{align*}
\end{lem}

\vskip2mm
\begin{lem}\label{Prem 03}
	Let us consider $s\geq0$ and $\alpha_2>0$. Then, the following estimate holds:
	\begin{align*}
	\left\|\chi_{\intt}(\xi)|\xi|^se^{-c|\xi|^{\alpha_2}t}\right\|_{L^2(\mb{R}^n)}\gtrsim (1+t)^{-\frac{2s+n}{2\alpha_2}},
	\end{align*}
	where the constant $c>0$.
\end{lem}
\begin{proof}
	By directly calculation, we obtain
	\begin{align*}
	\left\|\chi_{\intt}(\xi)|\xi|^se^{-c|\xi|^{\alpha_2}t}\right\|_{L^2(\mb{R}^n)}^2&=\int_{\mb{R}^n}\chi_{\intt}^2(\xi)|\xi|^{2s}e^{-2c|\xi|^{\alpha_2}t}d\xi\\
	&=\int_{0}^{\varepsilon}\int_{|\xi|=r}r^{2s}e^{-2cr^{\alpha_2}t}dS_{\xi}dr\\
	&=\omega_n\int_0^{\varepsilon}r^{2s+n-1}e^{-2cr^{\alpha_2}t}dr,
	\end{align*}
	where $\omega_n=\int_{|\omega|=1}d\omega=\frac{2\pi^{\frac{n}{2}}}{\Gamma(\frac{n}{2})}$. By using the ansatz $r^{\alpha_2}t=\tau$, we complete the proof of the lemma.
\end{proof}

\vskip2mm
\begin{thm}\label{Thm asym 01}
	Let us assume $U_0\in H^s(\mb{R}^n)\cap L^{1,1}(\mb{R}^n)$ with $|P_{U_0}|\neq0$, where $s\geq0$. Then, the solution $U=U(t,x)$ to the Cauchy problem \eqref{Eq. thermo plate Fourier} with $\sigma\in[0,1)$ satisfies the following estimates for $t\gg1$:
	\begin{align*}
	t^{-\frac{n+2s}{4(2-\sigma)}}\left|P_{U_0}\right|\lesssim\left\|U(t,\cdot)\right\|_{\dot{H}^s(\mb{R}^n)}\lesssim t^{-\frac{n+2s}{4(2-\sigma)}}\|U_0\|_{H^s(\mb{R}^n)\cap L^{1,1}(\mb{R}^n)}.
	\end{align*}
\end{thm}
\begin{proof}
	To begin with, let us define
	\begin{align*}
	J_4(t,|\xi|):=T_0T_1T_{1\frac{1}{2}}\diag\left(e^{-|\xi|^{4-2\sigma}t},e^{-|\xi|^2t},e^{-|\xi|^{2\sigma}t}\right)T_{1\frac{1}{2}}^{-1}T_1^{-1}T_0^{-1}.
	\end{align*}
	We use Lemmas \ref{Pre 01}, \ref{Pre 02} and Theorem \ref{Thm Diffusion Lm} to get
	\begin{align*}
	&\left\|\chi_{\intt}(D)\ml{F}^{-1}_{\xi\rightarrow x}\left(w^{(0)}\right)-\chi_{\intt}(D)\ml{F}_{\xi\rightarrow x}^{-1}\left(J_4(t,|\xi|)\right)P_{U_0}\right\|_{\dot{H}^s(\mb{R}^n)}\\
	&\qquad=\left\|\chi_{\intt}(D)\ml{F}^{-1}_{\xi\rightarrow x}\left(w^{(0)}-T_0T_1T_{1\frac{1}{2}}\tilde{w}\right)\right.\\
	&\qquad\qquad+\left.\chi_{\intt}(D)\ml{F}^{-1}_{\xi\rightarrow x}\big(J_4(t,|\xi|)\left(A_{U_0}(\xi)-iB_{U_0}(\xi)\right)\big)\right\|_{\dot{H}^s(\mb{R}^n)}\\
	&\qquad\lesssim (1+t)^{-\gamma(\sigma,n,1,s)-\frac{1-\sigma}{2-\sigma}}\left\|\left(|D|^2u_0,u_1,\theta_0\right)\right\|_{L^1(\mb{R}^n)\times L^1(\mb{R}^n)\times L^1(\mb{R}^n)}\\
	&\qquad\quad\,\,+\left\|\chi_{\intt}(\xi)|\xi|^{s+1}e^{-|\xi|^{4-2\sigma}t}\right\|_{L^2(\mb{R}^n)}\left\|\left(|D|^2u_0,u_1,\theta_0\right)\right\|_{L^{1,1}(\mb{R}^n)\times L^{1,1}(\mb{R}^n)\times L^{1,1}(\mb{R}^n)}\\
	&\qquad\lesssim (1+t)^{-\gamma(\sigma,n,1,s)-\frac{1-\sigma}{2-\sigma}}\left\|\left(|D|^2u_0,u_1,\theta_0\right)\right\|_{L^1(\mb{R}^n)\times L^1(\mb{R}^n)\times L^1(\mb{R}^n)}\\
	&\qquad\quad\,\,+(1+t)^{-\gamma(\sigma,n,1,s+1)}\left\|\left(|D|^2u_0,u_1,\theta_0\right)\right\|_{L^{1,1}(\mb{R}^n)\times L^{1,1}(\mb{R}^n)\times L^{1,1}(\mb{R}^n)}.
	\end{align*}
	To get the lower bounds estimates, we apply triangle inequality to obtain
	\begin{align*}
	&\left\|\chi_{\intt}(D)\ml{F}^{-1}_{\xi\rightarrow x}\left(w^{(0)}\right)\right\|_{\dot{H}^s(\mb{R}^n)}\\
	&\qquad\quad\geq\left\|\chi_{\intt}(D)\ml{F}_{\xi\rightarrow x}^{-1}\left(J_4(t,|\xi|)\right)P_{U_0}\right\|_{\dot{H}^s(\mb{R}^n)}\\
	&\qquad\quad\quad\,\,-\left\|\chi_{\intt}(D)\ml{F}^{-1}_{\xi\rightarrow x}\left(w^{(0)}\right)-\chi_{\intt}(D)\ml{F}_{\xi\rightarrow x}^{-1}\big(J_4(t,|\xi|)\big)P_{U_0}\right\|_{\dot{H}^s(\mb{R}^n)}\\
	&\qquad\quad\gtrsim\left\|\chi_{\intt}(\xi)|\xi|^s\left(e^{-|\xi|^{4-2\sigma}t}+e^{-|\xi|^2t}+e^{-|\xi|^{2\sigma}t}\right)\right\|_{L^2(\mb{R}^n)}\left|P_{U_0}\right|\\
	&\qquad\quad\quad\,\,-(1+t)^{-\gamma(\sigma,n,1,s)-\frac{1-\sigma}{2-\sigma}}\left\|\left(|D|^2u_0,u_1,\theta_0\right)\right\|_{L^1(\mb{R}^n)\times L^1(\mb{R}^n)\times L^1(\mb{R}^n)}\\
	&\qquad\quad\quad\,\,-(1+t)^{-\gamma(\sigma,n,1,s+1)}\left\|\left(|D|^2u_0,u_1,\theta_0\right)\right\|_{L^{1,1}(\mb{R}^n)\times L^{1,1}(\mb{R}^n)\times L^{1,1}(\mb{R}^n)}.
	\end{align*}
	In conclusion, for $t\gg1$ the following estimate holds:
	\begin{align*}
	\left\|\chi_{\intt}(D)\ml{F}^{-1}_{\xi\rightarrow x}\left(w^{(0)}\right)\right\|_{\dot{H}^s(\mb{R}^n)}\gtrsim t^{-\gamma(\sigma,n,1,s)}\left|P_{U_0}\right|.
	\end{align*}
	Combining with the upper bounds estimate for $t\gg1$ such that
	\begin{align*}
	\left\|\ml{F}^{-1}_{\xi\rightarrow x}\left(w^{(0)}\right)\right\|_{\dot{H}^s(\mb{R}^n)}\lesssim t^{-\gamma(\sigma,n,1,s)}\big\|U_0\|_{H^s(\mb{R}^n)\cap L^{1,1}(\mb{R}^n)},
	\end{align*}
	and
	\begin{align*}
	\begin{split}
	\left\|\chi_{\intt}(D)\ml{F}^{-1}_{\xi\rightarrow x}\left(w^{(0)}\right)\right\|_{\dot{H}^s(\mb{R}^n)}\lesssim\left\|\ml{F}^{-1}_{\xi\rightarrow x}\left(w^{(0)}\right)\right\|_{\dot{H}^s(\mb{R}^n)},
	\end{split}
	\end{align*}
	we complete the proof.
\end{proof}

\vskip2mm
\begin{thm}\label{Thm asym 02}
	Let us assume $U_0\in H^s(\mb{R}^n)\cap L^{1,1}(\mb{R}^n)$ with $|P_{U_0}|\neq0$, where $s\geq0$. Then, the solution $U=U(t,x)$ to the Cauchy problem \eqref{Eq. thermo plate Fourier} with $\sigma\in(1,2]$ satisfies the following estimates for $t\gg1$:
	\begin{align*}
	t^{-\frac{n+2s}{4}}\left|P_{U_0}\right|\lesssim\left\|U(t,\cdot)\right\|_{\dot{H}^s(\mb{R}^n)}\lesssim t^{-\frac{n+2s}{4}}\|U_0\|_{H^s(\mb{R}^n)\cap L^{1,1}(\mb{R}^n)}.
	\end{align*}
\end{thm}
\begin{proof}
	Following the proof of Theorem \ref{Thm asym 01} one can complete this proof.
\end{proof}

Finally, we derive asymptotic profiles of solutions to the Cauchy problem \eqref{Eq. thermo plate Fourier} for the case $\sigma=1$.

\vskip2mm
\begin{thm}\label{Thm asym 03}
	Let us assume $U_0\in H^s(\mb{R}^n)\cap L^{1,1}(\mb{R}^n)$ with $|P_{U_0}|\neq0$, where $s\geq0$. Then, the solution $U=U(t,x)$ to the Cauchy problem \eqref{Eq. thermo plate Fourier} with $\sigma=1$ satisfies the following estimates for $t\gg1$:
	\begin{align*}
	t^{-\frac{n+2s}{4}}\left|P_{U_0}\right|\lesssim\left\|U(t,\cdot)\right\|_{\dot{H}^s(\mb{R}^n)}\lesssim t^{-\frac{n+2s}{4}}\|U_0\|_{H^s(\mb{R}^n)\cap L^{1,1}(\mb{R}^n)}.
	\end{align*}
\end{thm}
\begin{proof}
	Let us define
	\begin{align*}
	J_5(t,|\xi|):=T_{1,0}\diag\left(e^{-y_4|\xi|^2t},e^{-y_5|\xi|^2t},e^{-y_6|\xi|^2t}\right)T_{1,0}^{-1}.
	\end{align*}
	From Theorem \ref{Thm. Rep. Sol. =1 freq. Fourier}, we know
	\begin{align*}
	&\left\|\chi_{\intt}(D)\ml{F}^{-1}_{\xi\rightarrow x}\left(w^{(0)}\right)-\chi_{\intt}(D)\ml{F}^{-1}_{\xi\rightarrow x}\left(J_5(t,|\xi|)\right)P_{U_0}\right\|_{\dot{H}^s(\mb{R}^n)}\\
	&\qquad\qquad=\left\|\chi_{\intt}(D)\ml{F}^{-1}_{\xi\rightarrow x}\left(J_5(t,|\xi|)\left(A_{U_0}(\xi)-iB_{U_0}(\xi)\right)\right)\right\|_{\dot{H}^s(\mb{R}^n)}\\
	&\qquad\qquad\lesssim(1+t)^{-\frac{n+2s}{4}-\frac{1}{2}}\|U_0\|_{L^{1,1}(\mb{R}^n)},
	\end{align*}
	where $T_{0,1}$, $y_4,y_5,y_6$ are defined in Lemma \ref{Lemma 2.3}.\\
	Then, repeating the procedure of the proof of Theorem \ref{Thm asym 01} we derive for $t\gg1$
	\begin{align*}
	t^{-\frac{n+2s}{4}}\left|P_{U_0}\right|\lesssim\left\|\chi_{\intt}(D)\ml{F}^{-1}_{\xi\rightarrow x}\left(w^{(0)}\right)\right\|_{\dot{H}^s(\mb{R}^n)}\lesssim \left\|\ml{F}^{-1}_{\xi\rightarrow x}\left(w^{(0)}\right)\right\|_{\dot{H}^s(\mb{R}^n)},
	\end{align*}
	and the proof of Theorem \ref{Thm asym 03} is complete.
\end{proof}
\vskip2mm
\begin{rem}
	According to Theorems \ref{Thm asym 01}, \ref{Thm asym 02} and \ref{Thm asym 03}, the thermal dissipation generated by Fourier's law has a dominant influence in comparison with structural damping on long-time asymptotic profiles of solutions when $\sigma\in[1,2]$.
\end{rem}
\section{Concluding remarks}\label{Section Concluding remarks}

\begin{rem}\label{Generalization}
	In general, our method to derive sharp asymptotic profiles of solutions in the framework of $L^{1,1}$ can be probably applied to the Cauchy problem for other systems in elastic material including elastic waves with different damping mechanisms, thermoelastic systems, thermodiffusion systems.\\
	In detail, for elastic waves with friction or structural damping \cite{Reissig2016,ChenReissig2018}, elastic waves with Kelvin-Voigt damping \cite{Chen2018}, thermoelastic systems \cite{JachmannReissig2008,WangYang2006,YangWang2006,YangWang200602,ReissigWang2005} and thermodiffusion systems \cite{LiuReissig2014}, the authors applied diagonalization procedures or asymptotic expansions of eigenvalues/eigenprojections to derive representations of solutions. By these representations of solutions, one may obtain diffusion phenomena with weighted $L^1$ data.  Then, one can follow the method in Section \ref{Section Asymptotic behavior}  to derive the sharp estimates for lower bounds and upper bounds of solutions in a framework of weighted $L^1$ data. 
\end{rem}
\subsection{Summary}
In the following we will collect results for the Cauchy problem for thermoelastic plate equations with friction or structural damping \eqref{Eq. thermo plate Fourier}.

In the paper we first derive Gevrey smoothing of solutions (see Table \ref{Table}) and $L^2$ well-posedness for the Cauchy problem \eqref{Eq. thermo plate Fourier} such that
\begin{align*}
U\in\ml{C}\big([0,\infty),L^2(\mb{R}^n)\big)\,\,\,\,\text{if we assume}\,\,\,\,U_0\in L^2(\mb{R}^n).
\end{align*}

Next, we obtain several decay estimates of solutions. On one hand, we derive the following energy estimates:
\begin{align*}
\|U(t,\cdot)\|_{\dot{H}^s(\mb{R}^n)}\lesssim(1+t)^{-\frac{(2-m)n+2ms}{2mK}}\left\|U_0\right\|_{H^s(\mb{R}^n)\cap L^m(\mb{R}^n)},
\end{align*}
where $s\geq0$, $m\in[1,2]$, and 
\begin{align*}
\|U(t,\cdot)\|_{\dot{H}^s(\mb{R}^n)}\lesssim(1+t)^{-\frac{n+2(s+\delta)}{2K}}\left\|U_0\right\|_{H^s(\mb{R}^n)\cap L^{1,\delta}(\mb{R}^n)}+(1+t)^{-\frac{n+2s}{2K}}\left|P_{U_0}\right|,
\end{align*}
where $s\geq0$, $\delta\in(0,1]$. Here some numbers $K$ (specified in the table below). On the other hand, there are $L^p-L^q$ estimates not necessary on the conjugate line of the form
\begin{align*}
\|U(t,\cdot)\|_{\dot{H}^s_q(\mb{R}^n)}\lesssim (1+t)^{-\frac{s}{K}-\frac{n}{K}\left(\frac{1}{p}-\frac{1}{q}\right)}\|U_0\|,
\end{align*}
for suitable $p,q$ and some numbers $K$ (specified in the table below). Here $\|U_0\|$ corresponds to initial data measured in an appropriate norm, which is based on $L^p$.

Finally, we derive diffusion phenomena with data belonging to different function spaces, and asymptotic profiles of solutions with weighted $L^{1}$ data
\begin{align*}
t^{-\frac{n+2s}{2K}}\left|P_{U_0}\right|\lesssim\left\|U(t,\cdot)\right\|_{\dot{H}^s(\mb{R}^n)}\lesssim t^{-\frac{n+2s}{2K}}\|U_0\|_{H^s(\mb{R}^n)\cap L^{1,1}(\mb{R}^n)}
\end{align*}
for $t\gg1$, where $|P_{U_0}|\neq0$,  $s\geq0$ and some numbers $K$ are chosen in Table \ref{Table}.

\begin{table}[http]
	\centering
	\resizebox{132mm}{25mm}{
		\begin{tabular}{ |m{1.8cm}<{\centering} ||c | c | c| c | c| c |}
			\hline
			& $\sigma=0$ & $\sigma\in(0,1)$ & $\sigma=1$ & $\sigma\in\left(1,3/2\right]$ & $\sigma\in\left(3/2,2\right)$ & $\sigma=2$ \\ 
			\hline
			\hline
			Gevrey smoothing  & \multicolumn{4}{c|}{$\Gamma^{1}(\mb{R}^n)$ (analytic smoothing)}    & $\Gamma^{\frac{1}{4-2\sigma}}(\mb{R}^n)$ & - \\
			\hline
			Energy estimates & \multicolumn{2}{c|}{\makecell[{}{p{2.6cm}}] {$\quad\,\,\,\, K=4-2\sigma$}} &\multicolumn{4}{c|}{\makecell[{}{p{5cm}}] {$\qquad \qquad   \,\,\,\, K=2$}} \\
			\hline
			$L^p-L^q$ estimates & \multicolumn{2}{c|}{\makecell[{}{p{2.6cm}}] {$\,\,\,K=4-2\sigma$ and \\$1\leq p\leq 2\leq q\leq\infty$}} & {\makecell[{}{p{2.3cm}}] {$\quad K=2$ and\\ $1\leq p\leq q\leq\infty$}} &\multicolumn{3}{c|}{\makecell[{}{p{2.7cm}}] {$\qquad K=2$ and\\ $1\leq p\leq 2\leq q\leq\infty$}} \\
			\hline
			Diffusion phenomena (dif. phe.)  & {\makecell[{}{p{1.3cm}}] {double dif. phe.}}& {\makecell[{}{p{1.3cm}}] {triple dif. phe.}} & - & \multicolumn{3}{c|}{\makecell[{}{p{2.7cm}}] {single dif. phe. }} \\
			\hline
			Asymptotic profiles & \multicolumn{2}{c|}{\makecell[{}{p{3.4cm}}] {$\qquad\,\, K=4-2\sigma$}} & \multicolumn{4}{c|}{\makecell[{}{p{5cm}}] {$\qquad \qquad \qquad \qquad \,\,\,\, K=2$}}\\
			\hline
	\end{tabular}}
	\caption{Summary for qualitative properties of solutions}
	\label{Table}
\end{table}
We should point out that when $K=4$, friction has a dominant influence in the corresponding decay estimates; when $K=4-2\sigma$, structural damping has a dominant influence in the corresponding decay estimates; when $K=2$, thermal dissipation generated by Fourier's law has a dominant influence in the corresponding decay estimates.
\subsection{Estimates for the solution itself}
Throughout this paper, we apply diagonalization procedure to get the representations of solutions
\begin{align}\label{solution diag}
U(t,x)=\left(u_t+|D|^2u,u_t-|D|^2u,\theta\right)^{\mathrm{T}}(t,x)
\end{align}
and study some qualitative properties of solutions to the Cauchy problem \eqref{Eq. thermo plate Fourier}.

Nevertheless, up to now, concerning the qualitative properties of the solution $u=u(t,x)$ to the Cauchy problem \eqref{Eq. thermo plate Fourier}, we did not derive any estimate for the solution itself. In this section, we will show some strategies to derive estimates for the solution itself. We propose three different strategies.
\vskip2mm

\noindent\emph{Strategy 1. } Estimates of $u$ by using the Riesz potential theory.
\vskip2mm
We formally define the Riesz potential in $\mb{R}^n$ by its action on a measurable function $f=f(x)$ by convolution, that is
\begin{align*}
\left(\tilde{I}_{2\kappa}f\right)(x)\equiv\left(\tilde{I}_{2\kappa}\ast f\right)(x):=\ml{F}^{-1}\left(|\xi|^{-2\kappa}\hat{f}(\xi)\right)(x)\equiv C_{n,\kappa}\int_{\mb{R}^n}\frac{f(y)}{|x-y|^{n-2\kappa}}dy,
\end{align*}
where $\kappa\in\left(0,\frac{n}{2}\right)$.

The study of the following mapping properties to $\tilde{I}_{2\kappa}$ was initiated by \cite{Sobolev1938}.

\vskip2mm
\begin{lem}\label{Riesz mapping}
	Let us assume $f\in L^{p}(\mb{R}^n)$ for $p\in\left(1,\frac{n}{2\kappa}\right)$. Then, $\tilde{I}_{2\kappa}f\in L^{p^*}(\mb{R}^n)$, where
	\begin{align*}
	\left\|\tilde{I}_{2\kappa}f\right\|_{L^{p^*}(\mb{R}^n)}\lesssim \|f\|_{L^{p}(\mb{R}^n)}\,\,\,\,\text{with}\,\,\,\,\frac{1}{p}-\frac{1}{p^*}=\frac{2\kappa}{n}.
	\end{align*}
\end{lem}
Then, we may estimates the solution itself by
\begin{align*}
\|u(t,\cdot)\|_{L^q(\mb{R}^n)}&\lesssim\left\|\ml{F}^{-1}_{\xi\rightarrow x}\left(|\xi|^{-2}w^{(0)}\right)(t,\cdot)\right\|_{L^q(\mb{R}^n)}=\left\|\tilde{I}_2\ml{F}^{-1}_{\xi\rightarrow x}\left(w^{(0)}\right)(t,\cdot)\right\|_{L^q(\mb{R}^n)}\\
&\lesssim\left\|\ml{F}^{-1}_{\xi\rightarrow x}\left(w^{(0)}\right)(t,\cdot)\right\|_{L^{\frac{2q+n}{nq}}(\mb{R}^n)},
\end{align*}
where $0<\frac{1}{q}<1-\frac{2}{n}$. Next, following a similar procedure of Section \ref{Subsec Fouier estimate}, one may complete estimates of the solution itself.
\vskip2mm

\noindent\emph{Strategy 2. } Estimates of $u$ by using the integral formula.
\vskip2mm
For the case $\frac{1}{q}\in[0,1]\backslash\left(0,1-\frac{2}{n}\right)$, we cannot apply \emph{Strategy 1}. Therefore, by the integral formula
\begin{align*}
u(t,x)-u(0,x)=\int_0^tu_{\tau}(\tau,x)d\tau,
\end{align*}
we obtain
\begin{align*}
\|u(t,\cdot)\|_{L^q(\mb{R}^n)}\lesssim\|u_0\|_{L^q(\mb{R}^n)}+\int_0^t\|u_{\tau}(\tau,\cdot)\|_{L^q(\mb{R}^n)}d\tau.
\end{align*}
Next, we apply the estimates of $u_{\tau}(\tau,\cdot)$ in the $L^q$ norm to complete estimates of the solution itself. We should remark that to apply this strategy, we need to take an additional assumption on the first data such that $u_0\in L^q(\mb{R}^n)$.
\vskip2mm

\noindent\emph{Strategy 3. } Estimates of $u$ by using the representation of the solution.
\vskip2mm
By some direct calculations, we may transfer the Cauchy problem \eqref{Eq. thermo plate Fourier} to the following Cauchy problem for third-order equation:
\begin{align}\label{Eq. third eq}
\left\{
\begin{aligned}
&u_{ttt}+(-\Delta)^{\sigma}u_{tt}-\Delta u_{tt}+2\Delta^2u_t+(-\Delta)^{\sigma+1}u_t-\Delta^3u=0,&&t>0,\,\,x\in\mb{R}^n,\\
&(u,u_t,u_{tt})(0,x)=(u_0,u_1,u_2)(x),&&x\in\mb{R}^n,
\end{aligned}
\right.
\end{align}
where $\sigma\in[0,2]$ and \begin{align*}
u_2(x):=-\Delta^2u_0(x)-(-\Delta)^{\sigma}u_1(x)-\Delta\theta_0(x).
\end{align*}
Applying the partial Fourier transformation with respect to spatial variables to \eqref{Eq. third eq}, we obtain an ordinary differential equation depending on the parameter $|\xi|$
\begin{align}\label{Eq. third eq Fourier}
\left\{
\begin{aligned}
&\hat{u}_{ttt}+\left(|\xi|^{2\sigma}+|\xi|^2\right) \hat{u}_{tt}+\left(2|\xi|^{4}+|\xi|^{2\sigma+2}\right)\hat{u}_t+|\xi|^6\hat{u}=0,&&t>0,\,\,\xi\in\mb{R}^n,\\
&(\hat{u},\hat{u}_t,\hat{u}_{tt})(0,\xi)=(\hat{u}_0,\hat{u}_1,\hat{u}_2)(\xi),&&\xi\in\mb{R}^n.
\end{aligned}
\right.
\end{align}
The characteristic roots $\lambda_j=\lambda_j(|\xi|)$, $j=1,2,3,$ for the equation of \eqref{Eq. third eq Fourier} satisfy the parameter dependent cubic equation
\begin{align}\label{Cubic equation}
\lambda^3+\left(|\xi|^{2\sigma}+|\xi|^2\right)\lambda^2+\left(2|\xi|^4+|\xi|^{2\sigma+2}\right)\lambda+|\xi|^6=0.
\end{align} 
We may find the exact solution of \eqref{Cubic equation} as follows:
\begin{align}\label{lambda}
\lambda_{j}=r_{3,j}-\frac{r_1}{3r_{3,j}}-\frac{|\xi|^{2\sigma}+|\xi|^2}{3},
\end{align}
where
\begin{align*}
r_1&=\frac{1}{3}\left(5|\xi|^4+|\xi|^{2\sigma+2}-|\xi|^{4\sigma}\right),\\
r_2&=\frac{1}{27}\left(11|\xi|^6+2|\xi|^{6\sigma}-3|\xi|^{4\sigma+2}-2|\xi|^{2\sigma+4}\right),\\
r_{3}^3&=\frac{1}{2}\left(-r_2\pm\sqrt{r_2^2+\frac{4}{27}r_1^3}\right),\,\,\,\,\text{where}\,\,\,\,r_{3,1},r_{3,1},r_{3,1}\text{ are complex solutions of it.}
\end{align*}
It provides an opportunity for us to derive explicit representations of solutions to \eqref{Eq. third eq Fourier} such that
\begin{align*}
\hat{u}(t,\xi)=c_1(\xi)e^{\lambda_1(|\xi|)t}+c_2(\xi)e^{\lambda_2(|\xi|)t}+c_3(\xi)e^{\lambda_3(|\xi|)t},
\end{align*}
where $\lambda_j(|\xi|)$ are given by \eqref{lambda} and the coefficients $c_j(\xi)$ are given by
\begin{small}
	\begin{align*}
	\begin{split}
	c_1(\xi)&=\frac{\lambda_2(|\xi|)\lambda_3(|\xi|)\hat{u}_0(\xi)-\lambda_2(|\xi|)\hat{u}_1(\xi)-\lambda_3(|\xi|)\hat{u}_1(\xi)-|\xi|^4\hat{u}_0(\xi)-|\xi|^{2\sigma}\hat{u}_1(\xi)+|\xi|^2\hat{\theta}_0(\xi)}{\lambda_1^2(|\xi|)-\lambda_1(|\xi|)\lambda_2(|\xi|)-\lambda_1(|\xi|)\lambda_3(|\xi|)+\lambda_2(|\xi|)\lambda_3(|\xi|)},\\
	c_2(\xi)&=\frac{\lambda_1(|\xi|)\lambda_3(|\xi|)\hat{u}_0(\xi)-\lambda_1(|\xi|)\hat{u}_1(\xi)-\lambda_3(|\xi|)\hat{u}_1(\xi)-|\xi|^4\hat{u}_0(\xi)-|\xi|^{2\sigma}\hat{u}_1(\xi)+|\xi|^2\hat{\theta}_0(\xi)}{\lambda_2^2(|\xi|)-\lambda_1(|\xi|)\lambda_2(|\xi|)+\lambda_1(|\xi|)\lambda_3(|\xi|)-\lambda_2(|\xi|)\lambda_3(|\xi|)},\\
	c_3(\xi)&=\frac{\lambda_1(|\xi|)\lambda_2(|\xi|)\hat{u}_0(\xi)-\lambda_1(|\xi|)\hat{u}_1(\xi)-\lambda_2(|\xi|)\hat{u}_1(\xi)-|\xi|^4\hat{u}_0(\xi)-|\xi|^{2\sigma}\hat{u}_1(\xi)+|\xi|^2\hat{\theta}_0(\xi)}{\lambda_3^2(|\xi|)+\lambda_1(|\xi|)\lambda_2(|\xi|)-\lambda_1(|\xi|)\lambda_3(|\xi|)-\lambda_2(|\xi|)\lambda_3(|\xi|)}.
	\end{split}
	\end{align*}
	
\end{small}
It is possible to derive estimates for $u=u(t,x)$ by applying these representations.

\vskip2mm

\section*{Acknowledgments} 
The PhD study of Mr. Wenhui Chen is supported by S\"achsiches Landesgraduiertenstipendium. The author thank his supervisor Michael Reissig for the suggestions in the preparation of the final version.

\bibliographystyle{elsarticle-num}
%\bibliography{References}

% ------------------------------------------------------------------------
\end{document}